\documentclass[12pt, notitlepage]{amsart}
\usepackage{amssymb}
\usepackage{amsmath}
\usepackage{enumitem}
\usepackage{mathrsfs}
\usepackage[all]{xy}
\usepackage{hyperref}
\usepackage{marginnote}

 \usepackage{array, graphicx}
\usepackage{pgf,tikz,amsfonts}

\usepackage{stmaryrd}

\usepackage[french,english]{babel}
\usepackage[utf8]{inputenc}

\usepackage{stmaryrd}

\usepackage[margin=1.0in]{geometry}

\RequirePackage{mathrsfs} 

\newtheorem{theorem}{Theorem}[section]
\newtheorem{lemma}[theorem]{Lemma}
\newtheorem{proposition}[theorem]{Proposition}
\newtheorem{corollary}[theorem]{Corollary}
\newtheorem{conjecture}[theorem]{Conjecture}

\theoremstyle{definition}
\newtheorem*{ack}{Acknowledgements}
\newtheorem*{con}{Conventions}
\newtheorem{remark}[theorem]{Remark}
\newtheorem{example}[theorem]{Example}
\newtheorem{definition}[theorem]{Definition}

\numberwithin{equation}{section} \numberwithin{figure}{section}

 \DeclareMathOperator{\NS}{NS}
\DeclareMathOperator{\Aut}{Aut}

\DeclareMathOperator{\Spec}{Spec}

\DeclareMathOperator{\an}{an}

\DeclareMathOperator{\Hom}{Hom}

\newcommand{\Qbar}{\overline{\QQ}}

\newcommand\PP{\mathbb{P}}

\newcommand\ZZ{\mathbb{Z}}

\newcommand\QQ{\mathbb{Q}}

\newcommand\CC{\mathbb{C}}

\usepackage{color}

\title[Demailly's notion of algebraic hyperbolicity]{Demailly's notion of algebraic hyperbolicity: Geometricity, boundedness, moduli of maps}

\author{Ariyan Javanpeykar}
\address{Ariyan Javanpeykar \\
Institut f\"{u}r Mathematik\\
Johannes Gutenberg-Universit\"{a}t Mainz\\
Staudingerweg 9, 55099 Mainz\\
Germany.}
\email{peykar@uni-mainz.de}

\author{Ljudmila Kamenova}
\address{Ljudmila Kamenova \\ 
Department of Mathematics \\
Stony Brook University \\ 
Stony Brook, NY 11794-3651 \\ 
USA.}
\email{kamenova@math.stonybrook.edu}



\begin{document}

\begin{abstract}
Demailly's conjecture, which is a consequence of  the  Green--Griffiths--Lang conjecture on varieties of general type, states that an algebraically hyperbolic complex projective variety is Kobayashi hyperbolic. Our aim is to provide evidence for Demailly's conjecture by verifying several predictions it makes. We first define what an algebraically hyperbolic projective variety is, extending Demailly's definition to (not necessarily smooth) projective varieties over an arbitrary algebraically closed field of characteristic zero, and we prove that this property is stable under extensions of algebraically closed fields. Furthermore, we show that the set of (not necessarily surjective) morphisms from a   projective variety $Y$  to a projective algebraically hyperbolic variety $X$ that map a  fixed closed subvariety of $Y$ onto a fixed   closed subvariety of $X$ is finite. As an application, we obtain that  $\Aut(X)$ is finite and that every surjective endomorphism of $X$ is an automorphism. Finally, we explore ``weaker'' notions of hyperbolicity related to boundedness of moduli spaces of maps, and verify similar predictions made by the Green--Griffiths--Lang conjecture on hyperbolic projective varieties.
\end{abstract}

\maketitle

\thispagestyle{empty}
 
\section{Introduction} 
The aim of this paper is to provide evidence for Demailly's conjecture which says that a projective algebraically hyperbolic variety over $\CC$ is Kobayashi hyperbolic.  

  We first define the notion of an algebraically hyperbolic projective scheme over an algebraically closed field  $k$ of characteristic zero which is not assumed to be $\mathbb{C}$, and  could be $\Qbar$, for example. Then we provide indirect evidence for Demailly's conjecture by showing that   algebraically hyperbolic schemes share many common features with    Kobayashi hyperbolic complex manifolds. 
  Furthermore, we also investigate ``weaker'' variants of algebraic hyperbolicity, and prove similar properties. Applications of our work are given in \cite{vBJK, Jaut,  JXie} (see \cite{JBook} for a survey).

\begin{definition} \label{def:alg_hyp}
A projective scheme $X$ over $k$ is \emph{algebraically hyperbolic over $k$}  if there is an ample line bundle $\mathcal{L}$, a positive real number $\alpha$, and a positive real number $\beta $ such that, for every smooth projective connected curve $C$ over $k$ and every $k$-morphism $f\colon C\to X$ we have that \[
\deg_C f^\ast \mathcal{L} \leq \alpha \cdot (2~ \mathrm{genus}(C) -2) + \beta = 
- \alpha \cdot \chi(C) +\beta.\]
\end{definition}

In \cite{Demailly} Demailly defines this notion for \emph{smooth} projective schemes over $\CC$ (and more generally, for compact complex manifolds and for projective directed manifolds). Note that the above definition   makes sense for 
(not necessarily smooth) projective schemes over $k$, and therefore extends Demailly's notion of algebraic hyperbolicity to singular varieties.

   If $f:C\to X$ is a non-constant morphism and $X$ is algebraically hyperbolic over $k$, then the genus of $C$ is greater than or equal to two, so that $-\alpha \cdot\chi(C) +\beta \leq (-\alpha -\beta) \cdot \chi(C)$.
We mention that $X$ is algebraically hyperbolic over $k$  if  and only if there is an integer  $g_0\geq 2$ and a positive real number $\alpha$ such that, for every smooth projective connected curve $C$ over $k$ of genus at least $g_0$ and every    morphism $f:C\to X$, the inequality $\deg_C f^\ast \mathcal{L} \leq \alpha \cdot \mathrm{genus}(C)$ holds.  Indeed, every curve $C$ of genus at least two has a finite \'etale covering $C'\to C$ with $C'$ a curve of genus at least $g_0$, and the degree of a morphism $C\to X$ is equal to the degree of the composition $C'\to C\to X$ divided by the degree of $C'\to C$.
 
Examples of algebraically hyperbolic projective varieties are given in \cite{Brotbek1, Brotbek2, Deb, Div, DivFer, Mourougane, RoulleauRousseau, Rousseau1}. Also, a logarithmic analogue of algebraic hyperbolicity (for quasi-projective varieties) was introduced and studied in \cite{ChenXi}.
 
 A finite type scheme $X$ over $\CC$ is Kobayashi hyperbolic if Kobayashi's pseudometric on the reduced complex analytic  space $X_{\mathrm{red}}^{\an}$ is a metric; see \cite{Kobayashi}.
The relation between algebraic hyperbolicity and  Kobayashi hyperbolicity   is provided by the following theorem of Demailly.  

\begin{theorem}[Demailly] \label{thm:dem1}  
If $X$ is a Kobayashi hyperbolic projective  scheme over $\CC$,  then $X$ is algebraically hyperbolic over $\CC$.
\end{theorem}

In \cite[Theorem~2.1]{Demailly}   Demailly shows that a Kobayashi hyperbolic \emph{smooth} projective variety over $\CC$ is algebraically hyperbolic (see also 
\cite[Theorem~2.13]{BKV}).   The smoothness assumption is however not used in Demailly's proof. 

Recall that a variety $X$ over $\CC$ is Brody hyperbolic if every holomorphic map $\mathbb{C}\to X^{\an}$ is constant, where $X^{\an}$ is the complex analytic space associated to $X$ \cite[Expos\'e~XII]{SGA1}.  Since Brody hyperbolic     projective varieties are Kobayashi hyperbolic  \cite{Kobayashi}, we see that Brody hyperbolic projective varieties over $\CC$ are algebraically hyperbolic over $\CC$. Similarly, as Borel hyperbolic projective varieties over $\CC$ (as defined in \cite{JKuch})  are Brody hyperbolic, it follows that they are also algebraically hyperbolic over $\CC$. In particular, roughly speaking, every ``complex-analytically'' hyperbolic variety is algebraically hyperbolic.
 
One can show that a projective Kobayashi hyperbolic variety $X$ over $\CC$ is \emph{groupless} (Definition \ref{defn:groupless}), i.e., for every connected complex algebraic group $G$, every morphism of varieties $G\to X$ is constant.  
In \cite[page 160]{Lang} Lang conjectured the converse, i.e.,   a groupless projective variety $X$ over $\CC$ is Kobayashi hyperbolic.
 
 Lang's aforementioned conjecture  is a variant of a similar conjecture of Green--Griffiths \cite{GrGr}. Indeed, Green and Griffiths conjectured that, if $X$ is a projective variety of general type over $\CC$, then there are no entire curves $\CC\to X^{\an}$ with Zariski dense image. Consequently, combining the conjectures of Lang and Green--Griffiths, we are led to the following conjecture (which we will refer to as the Green--Griffiths--Lang conjecture).
 
 \begin{conjecture}[Green--Griffiths--Lang]
 Let $X$ be a projective variety over $\CC$. Then the following are equivalent.
 \begin{enumerate}
 \item The projective variety $X$ is groupless over $\CC$.
 \item The complex analytic space $X^{\an}$ is Kobayashi hyperbolic.
 \item Every closed subvariety of $X$ is of general type.
 \end{enumerate}
 \end{conjecture}
 
 We now explain the relation of algebraically hyperbolic varieties to the Green--Griffiths--Lang conjecture. In fact, we  follow (and simplify) a strategy  of Demailly to show that projective algebraically hyperbolic varieties are groupless (Corollary \ref{cor:alg_hyp_is_groupless}). Therefore, as Demailly notes \cite[p.~14]{Demailly}, the converse of the statement of Theorem \ref{thm:dem1} is in fact a consequence of the Green--Griffiths--Lang conjecture. In other words, the following conjecture is a consequence of the Green--Griffiths--Lang conjecture (and we will refer to it as Demailly's conjecture throughout this paper).

\begin{conjecture}[Demailly, consequence of Green--Griffiths--Lang conjecture]\label{conj:demailly}
If $X$ is an algebraically hyperbolic projective variety over $\CC$, then $X$ is Kobayashi hyperbolic.
\end{conjecture} 

 In the next section we present our main results. We emphasize that all of our results are in accordance with Conjecture \ref{conj:demailly} in the sense that they allow one to verify some of the predictions one can make   \emph{assuming} Demailly's conjecture (Conjecture \ref{conj:demailly}) holds.

\subsection{Properties of algebraically hyperbolic varieties}\label{section:main_results}

Our first result illustrates that algebraic hyperbolicity is a 
geometric property. The proof is contained in Theorem \ref{thm1}. 
 
 \begin{theorem}\label{thm:geometricity} Let $k\subset L$ be an extension of algebraically closed fields of characteristic zero.
 Let $X$ be a projective algebraically hyperbolic scheme over $k$.  Then  the projective scheme $X_L$ is algebraically hyperbolic over $L$.
 \end{theorem}

The Green--Griffiths--Lang conjecture says that a projective  variety over $k$ is algebraically hyperbolic over $k$ if (and only if) it is groupless over $k$ (Definition \ref{defn:groupless}). Now, it is not hard to see that for $k\subset L$ an extension of algebraically closed fields of characteristic zero, a variety $X$ over $k$ is groupless over $k$ if and only if $X_L$ is groupless over $L$ (see Lemma \ref{lem:gr}). In particular, Theorem \ref{thm:geometricity} is in accordance with Green--Griffiths--Lang's aforementioned conjecture.

 The fact that the moduli of maps from \emph{any} given curve to an algebraically hyperbolic variety $X$ is bounded (by definition) has consequences for the moduli of maps from \emph{any} given variety to $X$, and also for the endomorphisms of $X$. The precise result we obtain reads as follows.

 \begin{theorem}\label{thm:aut_end}   Let $X$ be a projective algebraically hyperbolic   variety over $k$.   The following statements hold.
 \begin{enumerate}
\item   If $Y$ is a     projective reduced scheme over $k$, then the set of surjective morphisms   $Y\to X$ is finite. 
  \item  Assume that $X$ is reduced. Then, the group $\Aut(X)$ is finite, every  surjective endomorphism $X\to X$ of $X$ is an automorphism, and $X$ has only finitely many surjective endomorphisms. 
  \end{enumerate}
  \end{theorem}

The analogue of the first statement of Theorem \ref{thm:aut_end} for Kobayashi hyperbolic varieties was obtained by Noguchi; see \cite{Noguchi13} or  \cite[Theorem~6.6.2]{Kobayashi}. This latter statement (for Kobayashi hyperbolic varieties) was (also) conjectured by  Lang and has a long history; see \cite[\S6.6]{Kobayashi} for a discussion. For instance,  earlier results were obtained by Horst \cite{Horst4}. 
 An analogue of the statement about automorphisms for Kobayashi hyperbolic varieties is contained in \cite[Theorem~5.4.4]{Kobayashi}, and an analogue of the statements about endomorphisms for Kobayashi hyperbolic varieties is an application of \cite[Theorem~6.6.20]{Kobayashi} and \cite[Theorem~5.4.4]{Kobayashi}. Thus, needless to emphasize, we see that Theorem \ref{thm:aut_end} is in accordance with Demailly's conjecture (Conjecture \ref{conj:demailly}).

\begin{remark} In \cite[Theorem~3.5]{BKV}, the finiteness of $\Aut(X)$ is proven when $X$ is a \emph{smooth} projective algebraically hyperbolic variety over $\CC$. We stress that we do not impose smoothness. Moreover, we  allow for the base field to be any algebraically closed field of characteristic zero. Our proof of Theorem \ref{thm:aut_end} is different than the proof in \emph{loc. cit.} and allows for   a more general result (see Theorem \ref{thm:pointed_hom2}).
\end{remark}

The finiteness results in Theorem \ref{thm:aut_end} for surjective morphisms from a projective scheme to an algebraically hyperbolic projective scheme can in fact be subsumed into the following statement (which we prove using Theorem \ref{thm:aut_end}). 

\begin{theorem}\label{thm:pointed_hom}   
Let $X$ be an algebraically hyperbolic projective scheme over $k$. Then, for every projective scheme $Y$ over $k$, every non-empty reduced closed subscheme $B\subset Y$, and every reduced closed subscheme $A\subset X$, the set of morphisms $f:Y\to X$ with $f(B)=A$ is finite. 
 \end{theorem}
 
 The  analogue of Theorem \ref{thm:pointed_hom} for Kobayashi hyperbolic varieties when $\dim B = \dim A =0$ is Urata's theorem (see   \cite[Theorem~5.3.10]{Kobayashi} or the original paper \cite{Urata}).  Also, the analogue of the  statement of Theorem \ref{thm:pointed_hom} for Kobayashi hyperbolic  varieties is contained in \cite[Corollary~6.6.8]{Kobayashi}.  Thus, needless to stress, Theorem \ref{thm:pointed_hom} is also in accordance with Demailly's conjecture (Conjecture \ref{conj:demailly}). 
 
To state our following result, for $X$ and $Y$ projective schemes over $k$, we let $\underline{\Hom}_k(X,Y)$ be the associated Hom-scheme (see Section \ref{S2}). Moreover, we let $\underline{\Hom}_k^{nc}(X,Y)$ be the subscheme parametrizing   non-constant morphisms $X\to Y$.  

Roughly speaking, our next result verifies that moduli spaces of maps to a projective algebraically hyperbolic variety are (also) projective and algebraically hyperbolic.

\begin{theorem}\label{thm:projectivity_of_hom}  Let $X$ be a projective algebraically hyperbolic variety over $k$. If $Y$ is a projective scheme over $k$, then  the scheme $\underline{\Hom}_k(Y,X)$ is a projective algebraically hyperbolic scheme over $k$. Moreover, we have that $\dim\underline{\Hom}_k^{nc}(Y,X) <  \dim X$.
\end{theorem}

The analogue of Theorem \ref{thm:projectivity_of_hom} for Kobayashi hyperbolic projective varieties over $\CC$  is provided by
 \cite[Theorem~5.3.9]{Kobayashi} and \cite[Theorem~6.4.1]{Kobayashi}. Thus, like the two results above, Theorem \ref{thm:projectivity_of_hom}  is also in accordance with Demailly's conjecture (Conjecture \ref{conj:demailly}). 

 In the hope of understanding what properties of a projective scheme are sufficient for the conclusions of 
   Theorems   \ref{thm:geometricity}, \ref{thm:aut_end},   \ref{thm:pointed_hom}, and \ref{thm:projectivity_of_hom} to hold, we also investigate ``weaker'' notions of hyperbolicity.

 \subsection{Weaker notions of boundedness}\label{section:further_results} The results in the previous section were motivated by Demailly's conjecture (Conjecture \ref{conj:demailly}). With a view towards Green--Griffiths--Lang's more general conjecture, we seek for analogues of the results in Section \ref{section:main_results} for ``weaker'' notions of (algebraic) hyperbolicity. 
 
\begin{definition}
A projective scheme $X$ over $k$ is \emph{$1$-bounded over $k$} if for every smooth projective connected curve $C$ over $k$, the scheme $\underline{\Hom}_k(C,X)$ is of finite type over $k$.
\end{definition}

Note that, if $\mathcal{L}$ is an ample line bundle on a projective scheme $X$ over $k$, then $X$ is $1$-bounded over $k$ if and only if, for every smooth projective connected curve $C$ over $k$, there is a real number $\alpha_C$ (which depends only on $X$ and $C$) such that, for every morphism $f:C\to X$ the inequality  $\deg_C f^\ast \mathcal{L} \leq \alpha_C$ holds. 

Clearly, an algebraically hyperbolic projective scheme over $k$ is $1$-bounded over $k$. 
The ``difference'' between algebraic hyperbolicity and $1$-boundedness is in the uniformity of the bound we demand on the degree of a morphism $f:C\to X$. For $X$ to be algebraically hyperbolic, we demand $\deg_C f^\ast L$ to be bounded \textbf{linearly} in the \textbf{genus} of $C$. For $X$ to be $1$-bounded, we ask the latter to be bounded by a real number depending only on $C$.
 
 Despite the clear difference in the definitions, it seems reasonable to suspect that a $1$-bounded projective variety is algebraically hyperbolic over $k$. As we explain in Section \ref{S10},   the Green--Griffiths--Lang conjecture in fact predicts that $1$-bounded projective schemes over $k$ are algebraically hyperbolic over $k$. The results in this section are motivated by this latter observation. 

 We first show that $1$-boundedness is also a ``geometric'' property, i.e., it persists over any algebraically closed field extension of $k$.
 
 \begin{theorem}\label{thm:geometricity2} Let $k\subset L$ be an extension of algebraically closed fields of characteristic zero.
 Let $X$ be a projective $1$-bounded scheme over $k$.  Then  the projective scheme $X_L$ is $1$-bounded over $L$.
 \end{theorem}

Note that Theorems \ref{thm:aut_end} and \ref{thm:pointed_hom} follow from the following more general result.

\begin{theorem}\label{thm:pointed_hom2}   
Let $X$ be a $1$-bounded projective scheme over $k$. Then, for every projective scheme $Y$ over $k$, every non-empty reduced closed subscheme $B\subset Y$, and every reduced closed subscheme $A\subset X$, the set of morphisms $f:Y\to X$ with $f(B)=A$ is finite. 
 \end{theorem}

 Furthermore, analogous to Theorem \ref{thm:projectivity_of_hom}, we prove the following statement.

\begin{theorem}\label{thm:projectivity_of_hom2}   Let $X$ be a projective $1$-bounded scheme over $k$. If $Y$ is a projective scheme over $k$, then  the scheme $\underline{\Hom}_k(Y,X)$ is a projective $1$-bounded scheme over $k$. Moreover, we have that $\dim\underline{\Hom}_k^{nc}(Y,X) <  \dim X$.
\end{theorem}

 \subsection{From boundedness to uniform boundedness}
 It $X$ is an algebraically hyperbolic projective variety, then it follows from the definitions that $X$ is   $1$-bounded. Conjecturally, a $1$-bounded projective variety should actually be algebraically hyperbolic. The    \emph{a priori} difference is, as explained above, the nature of the bound we demand for the degree of a map in the  definitions. Our final  result provides a first step towards showing that a bounded projective scheme is algebraically hyperbolic.
 
 \begin{theorem}[Towards algebraic hyperbolicity]\label{thm:boundedness_to_uniform}
 Let $X$ be a $1$-bounded projective scheme over $k$. Then, for every ample line bundle $\mathcal{L}$ and every integer $g\geq 0$, there is an integer $\alpha(X,\mathcal{L},g)$ such that, for every smooth projective connected curve $C$ of genus $g$ over $k$ and every morphism $f:C\to X$, the inequality
 \[
 \deg_C f^\ast \mathcal{L} \leq \alpha(X,\mathcal{L},g)
 \] holds.
 \end{theorem}
 
 Thus,  the existence of a bound on the degree which is \emph{uniform in the genus}   is a consequence of the  existence of a  bound which could depend on the curve itself.  To prove   Theorem \ref{thm:boundedness_to_uniform} we use Theorem \ref{thm:projectivity_of_hom2}, and apply it to the universal stable curve of some fixed genus.

 \subsection{Applications and further results}
 In \cite{vBJK} we obtain further results on bounded and algebraically hyperbolic varieties. For example, we show that these notions are stable under generization. A survey of the results obtained here and in \cite{vBJK} is presented in \cite{JBook}.
 
As a further application of the results of this paper, we mention that  Theorem \ref{thm:geometricity} is used to prove that   an ``arithmetically hyperbolic'' projective variety over $k$ which is algebraically hyperbolic over $k$ remains arithmetically hyperbolic over any field extension of $k$; see \cite[Theorem~4.2]{Jaut} for a precise statement. Related results are also obtained in  \cite{JLev, JLitt}.    Moreover, one can also prove    ``arithmetic'' analogues of some of our results. For instance, the arithmetic analogues of the results in Section \ref{S5} are obtained in \cite[\S4]{JLar}. Also,  in \cite[Theorem~1.2]{Jaut}, we  prove that arithmetically hyperbolic projective varieties have only finitely many automorphisms,  and thereby obtain the arithmetic analogue of  Theorem \ref{thm:aut_end}.(2). We push these arithmetic results further in \cite{JXie}, and also investigate ``pseudo'' versions of the results of this paper in \emph{loc. cit.}.

 \subsection{Outline of paper} 
In Section \ref{S2} and Section \ref{S3} we gather some well-known results. For instance, we introduce the notion of   groupless varieties, and note that  a proper groupless variety has a countable discrete group of automorphisms. We combine this with a theorem of Hwang-Kebekus-Peternell to  prove that  the scheme $\underline{\mathrm{Sur}}(Y,X)$ parametrizing surjective morphisms from a projective variety $Y$ to a projective groupless variety $X$ over $k$ is a countable union of zero-dimensional projective schemes over $k$; see Theorem  \ref{thm:HKP}. Furthermore, 
in Section \ref{S3} we explore basic properties of projective varieties with no rational curves. We show that   finite type  components of certain $\Hom$-schemes of such varieties are proper, and that the ``evaluation maps'' defined on these $\Hom$-schemes are finite morphisms; see Corollary \ref{cor:crit_for_qf}. 

In Section \ref{section:boundedness} we study various (new) notions of boundedness and we explore   some relations between boundedness and  grouplessness.
Then, 
 in Section \ref{S5} we show that algebraic hyperbolicity, Kobayashi hyperbolicity, and all the various notions of   boundedness introduced in Section \ref{section:boundedness} behave similarly along finite    maps. Finally, the finiteness of the set of dominant rational maps, surjective endomorphisms and automorphisms of a bounded scheme are proven in Section \ref{S6}. 
 
 The geometricity of algebraic hyperbolicity (as predicted by the Green--Griffiths--Lang conjecture) is verified in Section \ref{S7}. In fact, we also prove that  every notion of boundedness introduced in this paper is ``geometric'' (i.e., persists over any algebraically closed field extension of the base field). 
 
 In Section \ref{S8} we use a Bertini-type argument to show that finiteness of pointed Hom-sets from curves implies finiteness of pointed Hom-sets from varieties; see Theorem \ref{thm:1m_is_nl} for a precise statement.  Similarly, in Section \ref{S9} we prove that boundedness of Hom-schemes from curves implies boundedness of Hom-schemes from varieties using a specialization argument (Theorem \ref{thm:1_implies_n}). As an application of this result,  we deduce that projective algebraically hyperbolic schemes are bounded (Theorem \ref{thm:1_implies_n}). This latter result implies, in particular,  that all the properties proven for bounded schemes in Sections \ref{section:boundedness}, \ref{S5} and \ref{S6} hold for algebraically hyperbolic schemes. 
 
We  prove all the results stated in the introduction in  Section \ref{S10}. In Section \ref{section:conjectures} we conclude the paper with  several conjectures related to Demailly's and Green--Griffiths--Lang's conjectures. These conjectures relate the various notions of boundedness introduced in this paper.

\begin{ack} We are grateful to Martin Olsson for helpful discussions and for suggesting we use   the moduli-theoretic arguments  employed in Section \ref{S7}. We are grateful to Jason Starr for telling us about the work of Hwang-Kebekus-Peternell which we used to prove Theorem \ref{thm:HKP}. We are grateful to Daniel Litt for many helpful discussions on purity.  We thank Eric Riedl for a helpful discussion in Montreal on Theorem \ref{thm:boundedness_to_uniform}. We are grateful to the referee for a careful reading of our paper, and for several helpful comments and suggestions (especially Remark \ref{referee}). We thank  Jarod Alper, Jean-Pierre Demailly, and Johan de Jong for  helpful discussions.
This research was supported through the programme  ``Oberwolfach Leibniz Fellows'' by the Mathematisches Forschungsinstitut Oberwolfach in 2018. The first named author gratefully acknowledges support from SFB Transregio/45.  The second named author is partially supported by a grant from the 
Simons Foundation/SFARI (522730, LK). 
\end{ack}

\begin{con}
Throughout this paper, we let $k$ be an algebraically closed field of characteristic zero. A variety over $k$ is a finite type separated reduced $k$-scheme.
\end{con}

\section{Grouplessness} \label{S2}
A variety is ``groupless'' if it does not admit any non-trivial map from an algebraic group.   The precise definition reads as follows.
\begin{definition}\label{defn:groupless}
A finite type scheme $X$ over $k$ is \emph{groupless (over $k$)} if, for every finite type connected group scheme $G$ over $k$, every morphism of $k$-schemes $G\to X$ is constant.
\end{definition}

Note that Kov\'acs \cite{KovacsSubs} and Kobayashi \cite[Remark~3.2.24]{Kobayashi} refer to groupless varieties as being ``algebraically hyperbolic'', and Hu--Meng--Zhang refers to groupless varieties as being ``algebraically Lang hyperbolic'' \cite{ShangZhang}. We  avoid this unfortunate mix of terminology, and only use the term ``algebraically hyperbolic'' in the sense of Demailly (Definition \ref{def:alg_hyp}).

The main result of this section is that the moduli space of surjective maps from a given projective variety to a given groupless projective variety is zero-dimensional; see Theorem \ref{thm:HKP} for a precise statement. We also take the opportunity to prove certain basic properties of groupless varieties.

\begin{lemma}\label{lem:ge} Let $S$ be an integral variety over $k$ with function field $K$. Let $K(S)\subset L$ be an algebraically closed field extension.
Let $X\to S$ be a morphism of varieties over $k$. Suppose that the set of $s$ in $S(k)$ such that   $X_s$ is groupless over $k$ is Zariski-dense in $S$. Then $X_{L}$ is groupless over $L$. 
\end{lemma}
\begin{proof}  Suppose that $X_{L}$ is not groupless. Then, we may choose  a $K$-finitely generated subfield $K(S)\subset K\subset L$,   a finite type connected group scheme $G$ over $K$, and a   non-constant morphism $G\to X_K$. Let $U$ be an integral variety with $K(U) = K$ and let $U\to S$ be a smooth dominant morphism of varieties over $k$ extending the inclusion $K(S)\subset K$. Let $\mathcal{G}$ be a finite type geometrically connected group scheme over $U$, and let $\mathcal{G}\to X\times_S U$  be a morphism of $U$-schemes which extends the morphism    $G\to X_K$ on the generic fibre. Our assumption that the set of $s$ in $S(k)$ with $X_s$ groupless is Zariski dense in $S$ implies that  the set of $s$ in $U(k)$ such that   $X_s$ is groupless is Zariski-dense in $U$. In particular, for a dense set of  $s$ in $U(k)$,  the morphism $\mathcal{G}_s\to (X\times_S U)_s = X_s$ is constant. This implies that the morphism $G\to X_K$ is constant, contradicting our assumption. We conclude that  $X_{L}$  is groupless over $L$.
\end{proof}

\begin{lemma}[Grouplessness is a geometric property]\label{lem:gr} Let $k\subset L$ be an extension of algebraically closed fields of characteristic zero and $X$ a finite type scheme over $k$. If $X$   is groupless over $k$, then $X_L$ is groupless over $L$.
\end{lemma}
\begin{proof} This follows from Lemma \ref{lem:ge} with $S=\Spec k$.
\end{proof}

\begin{remark}
A more general ``generizing'' property for grouplessness can be proven using non-archimedean methods; see  \cite[Theorem~1.3]{JV};
\end{remark}

 \begin{lemma}\label{lem:groupless}
 Let  $X$  be a finite type scheme over $k$. The following statements are equivalent.
 \begin{enumerate}
 \item  The finite type scheme $X$ is groupless over $k$.
 \item  Every morphism to $X$ from either  $\mathbb{G}_{m,k}$ or an abelian variety over $k$ is constant.
 \end{enumerate}
 \end{lemma}
 \begin{proof}
 That $(1)$ implies $(2)$ is clear. The other implication is a consequence of the structure theory of connected finite type (smooth quasi-projective geometrically connected) group schemes over $k$ \cite{ConradChevalley}. Indeed, assume $(2)$ holds. Let $G$ be a connected finite type group scheme over $k$. Let $H$ be the unique normal connected affine (closed) subgroup of $G$ such that $G/H$ is an abelian variety. Since any morphism $G/H\to X$ is constant, it suffices to show that any morphism $H\to X$ is constant. Let $U\subset H$ be the unipotent radical. Since any morphism $\mathbb{G}_{m,k}\to X$ is constant, we see that every morphism $\mathbb{G}_{a,k}\to X$ is constant. Therefore,  any morphism $U\to X$ is constant. Thus, we may replace by $H$ by $H/U$, so that $H$ is reductive.  However, since $H$ is the union of its Borel subgroups, we may and do assume that $H$ is a solvable group in which case it is clear that $H\to X$ is constant by $(2)$.
 \end{proof}

\begin{remark}
The proof of Lemma \ref{lem:groupless} shows that a complex analytic space $X$ is Brody hyperbolic if and only if, for every finite type connected group scheme $G$ over $\CC$, every holomorphic map $G^{\an}\to X$ is constant. (Similar statements are  true in a non-archimedean setting; see \cite{JV}.) In particular, one could refer to Brody hyperbolicity as ``$\CC$-analytic grouplessness''. Similarly, one could also refer to groupless varieties as ``algebraically-Brody hyperbolic varieties''.  
\end{remark}

\begin{lemma}\label{lem:groupless_for_proper}
Let $X$ be a proper variety over $k$. Then $X$ is groupless over $k$ if and only if, for every abelian variety $A$ over $k$, every morphism $A\to X$ is constant.
\end{lemma}
\begin{proof}
Suppose that, for every abelian variety $A$ over $k$, every morphism $A\to X$ is constant. To show that $X$ is groupless, it suffices to show that every morphism $\mathbb{G}_{m,k}\to X$ is constant (Lemma \ref{lem:groupless}). However, any such map extends to a morphism $\mathbb{P}^1_k\to X$. Let $E$ be an elliptic curve over $k$ and let $E\to \mathbb{P}^1_k$ be a surjective morphism. Then the composed morphism $E\to \mathbb{P}^1_k\to X$ is constant (by assumption), so that the morphism $\mathbb{P}^1_k\to X$ is constant. This proves the lemma.
\end{proof}

If $X$ and $Y$ are projective schemes over $k$, then the functor $\underline{\Hom}_k(X,Y)$ parametrizing morphisms $X\to Y$ is representable by  a countable disjoint union of quasi-projective schemes over $k$ (\cite[Section 4.c, pp. 221-19 -- 221-20]{Groth-FGA}). 
If $X$ is a projective scheme over a field $k$, then the functor $ \mathrm{Aut}_{X/k}$ parametrizing automorphisms of $X$ over $k$ is representable by a locally finite type separated group scheme over $k$ (which we also denote by $ \mathrm{Aut}_{X/k}$). 
If $X$ is a projective groupless scheme over $k$, then   $ \mathrm{Aut}_{X/k}$ is a zero-dimensional scheme, as we show now.  

 \begin{lemma}\label{lem:gr_has_discrete_aut}
 Let $X$ be a proper   variety over $k$. If $X$ is groupless, then $\mathrm{Aut}^0_{X/k}$ is trivial. In particular, the group $\mathrm{Aut}(X)$ is a countable discrete group.
 \end{lemma}
 \begin{proof} Let $G=\mathrm{Aut}^0_{X/k}$. This is a finitely presented connected group algebraic space over $k$; see  \cite[Theorem~6.1]{ArtinI}.
Thus, it is a finite type connected group scheme over $k$ (see \cite[Tag~06E9]{stacks-project}).
Since $X$ is groupless and $G$ is a finite type connected group scheme over $k$,
 for $x$ in $X$, the morphism $G\to X$  defined by $g\mapsto gx$ is constant. In other words, since every $g$ in $G$ acts trivially, we see that $G$ is the trivial group.
 \end{proof}

We can combine Lemma \ref{lem:gr_has_discrete_aut} with a   theorem of Hwang-Kebekus-Peternell to get a  stronger conclusion. 
  To state it, for    $Y$ is a projective scheme over $k$, recall that the functor $\underline{\mathrm{Sur}}(Y,X)$ parametrizing surjective morphisms from $Y$ to $X$ is representable by a locally finite type separated scheme over $k$. Indeed, it is an open subscheme of $\underline{\Hom}_k(Y,X)$.   
\begin{theorem}[Hwang-Kebekus-Peternell]\label{thm:HKP} Let $X$ be a projective groupless variety over $k$. Let $Y$ be a normal projective variety over $k$. Then  $\underline{\mathrm{Sur}}(Y,X)$   is a countable union of zero-dimensional projective schemes  over $k$.
\end{theorem}
\begin{proof}  Let $f:Y\to X$ be a surjective morphism. Let $\Hom_f(Y,X)$ be the connected component of $\underline{\Hom}(Y,X)$ containing $f$.
By  \cite[Theorem~1.2]{HKP}, as $X$ does not have any rational curves, there exists a factorization $Y\to Z\to X$ with $Z\to X$ a finite morphism, and a surjective morphism $\Aut^0_{Z/k} \to \underline{\Hom}_f(Y,X)$. Now, since $X$ is groupless, it follows that $Z$ is groupless, and thus $\Aut_{Z/k}^0$ is trivial. This implies that $\underline{\Hom}_f(Y,X)$ is a point. Therefore, since $\underline{\Hom}_k(Y,X)$ is a countable union of finite type schemes over $k$, this  concludes the proof.
\end{proof}

\begin{remark}
The ``converse'' of the theorem of Hwang-Kebekus-Peternell is not true. 
   Let $C$ be a smooth projective  curve of genus two, and let $X$ be the blow-up of $C\times_k C$ in a point. Then $X$ is a  smooth projective surface of general type. Note that $X$ is not groupless (as it contains a rational curve).
However, for any projective variety $Y$ over $k$, the set of surjective morphism  $Y\to X$ is finite.
\end{remark}

\section{Projective varieties with no rational curves} \label{S3}
 It turns out that, if $Y$ is a projective scheme over $k$ and $X$ is a projective scheme over $k$ with no rational curves, then the scheme $\underline{\mathrm{Hom}}_k(Y,X)$ is a countable union of \emph{projective} schemes; this is a well-known ingredient in Mori's ``bend-and-break''.
In this section we   collect some related  well-known results.
\begin{definition}\label{def:purity}
A variety $X$ over $k$ is \emph{pure (over $k$)} if, for every smooth variety $T$ over $k$ and every dense open $U\subset T$ with $\textrm{codim}(T\setminus U)\geq 2$, we have that every morphism $U\to X$ extends (uniquely) to a morphism $T\to X$.
\end{definition}

The notion of pure variety is also used (and extended) in \cite{vBJK} and \cite[Section~7.2]{JBook}. Note that a \emph{complete Kobayashi hyperbolic} variety over $\mathbb{C}$ is pure over $\mathbb{C}$; see \cite[Corollary~6.2.4]{Kobayashi}.

  \begin{lemma} \label{lem:rat_is_mor} 
Let $X$ be a proper pure variety over $k$. Let $Y$ be a smooth   variety over $k$. Then every rational dominant map from $Y$ to $X$ extends to a morphism $Y\to X$.     
\end{lemma}
\begin{proof} Since $X$ is proper, every rational dominant map from $Y$ to $X$ can be defined on an open $U\subset Y$ with $\mathrm{codim}(Y\setminus U) \geq 2$. Therefore, 
the lemma follows from  the definition of a pure variety (Definition \ref{def:purity}).
\end{proof}

\begin{remark}\label{remark:pure} Let $X$ be a variety over $k$.  Let $k\subset L$ be a field extension with $L$ algebraically closed. 
 Then $X$ is pure over $k$ if and only if $X_L$ is pure over $L$.
This follows from a standard spreading out and specialization argument (similar to the argument in the proof of Lemma \ref{lem:ge}).  
\end{remark}

\begin{lemma}\label{lem:hartog}
Let $X\to Y$ be an affine morphism of varieties over $k$. If $Y$ is pure over $k$, then $X$ is pure over $k$.
\end{lemma}
\begin{proof}
 This is a consequence of  Hartog's lemma.
\end{proof}

\begin{lemma}\label{lem:pure_means_no_rat_curve}
A proper variety $X$ over $k$ is pure if and only if it has no rational curves, i.e., every morphism $\mathbb{P}^1_k\to X$ is constant.
\end{lemma}
\begin{proof} Let $0  = (0:0:1) \in \mathbb{P}^2(k)$.
Since the projection   $\mathbb{P}^2_k\setminus \{0\}\to \PP^1$ does not extend to a morphism $\mathbb{P}^2\to \mathbb{P}^1$, we see that $\PP^1$ is not pure. Therefore, as a non-constant morphism $\PP^1_k\to X$ is finite (hence affine), a proper variety with a rational curve is not pure by Lemma \ref{lem:hartog}. Conversely, it follows from \cite[Proposition~6.2]{GLL} that a proper variety with no rational curves is pure.  
\end{proof}

\begin{example}
If $X$ is an affine variety over $k$, then $X$ is pure. Abelian varieties over $k$ are pure. An algebraic K3 surface over $k$ is not pure (as it contains a rational curve). 
\end{example}
The relation between groupless varieties and pure varieties is provided by the following proposition.  
 \begin{proposition}\label{prop:groupless_implies_pure}
 If $X$ is a   proper groupless   variety   over $k$, then $X$ is pure over $k$. 
 \end{proposition}
 \begin{proof}
Since proper groupless varieties have no rational curves, the proposition follows from Lemma \ref{lem:pure_means_no_rat_curve}.
 \end{proof}
  
 Note that a smooth proper genus one curve over $k$ is pure, but not groupless. Thus, there are pure smooth projective   varieties over $k$ which are not groupless.
  
 We show now that the rigidity lemma implies that    ``evaluation maps'' restricted to closed subschemes of certain pieces of $\Hom$-schemes    are finite.
\begin{lemma}\label{lem:applying_rigidity}
Let $X$ be a proper variety over $k$ and let $Y$ be a   projective variety over $k$. Let $Z$ be a locally closed subscheme of $\underline{\Hom}_k(Y,X)$. Assume that $Z$ is proper over $k$. Then, for any $y$ in $Y(k)$, the evaluation morphism \[
Z\to X, f\mapsto f(y)
\]
is  finite.
\end{lemma}
\begin{proof}
This is an application of the rigidity lemma \cite[Chapter~II]{MumAb} (cf. the argument of the proof of \cite[Corollary~5.3.4]{Kobayashi}). To be more precise, let $y$ be a point in $Y(k)$ and $x$ a point in $X(k)$. Note that the fibre over $x$ of the evaluation map $\mathrm{eval}_y:Z\to X$ defined by $\mathrm{eval}_y(f) = f(y)$ is the set $\Hom_k((Y,y),(X,x))\cap Z(k)$ of morphisms $f:Y\to X$ in $Z(k)$ with $f(y) = x$. To show that $\Hom_k((Y,y),(X,x))\cap Z(k)$ is finite, consider the closed subscheme $\underline{\Hom}_k((Y,y),(X,x))\subset \underline{\Hom}_k(Y,X)$ parametrizing maps $f:Y\to X$ with $f(y) = x$. Let $H$ be a connected component of $Z\cap\underline{\Hom}_k((Y,y),(X,x))$. Since $Z$ is proper, the scheme $H$ is proper over $k$. It suffices to show that $H(k)$ is a singleton. 

The morphism $\mathrm{eval}:Y\times H\to X$ given by $(y',f)\mapsto f(y')$ has the property that $(y,f) = x$ for all $f\in H$, i.e., it contracts $\{y\}\times H$ to a point. Thus, since $H$ is proper, the rigidity lemma implies that the morphism $\mathrm{eval}:Y\times H\to X$ factors over some morphism $g:Y\to X$, i.e., $\mathrm{eval} = g\circ \mathrm{pr}_Y$. In other words, for any $f$ in in $H$ and any $y$ in $Y$, we have that $f(y) = g(y)$. Thus, $H(k) = \{g\}$. This concludes the proof.
\end{proof}

To apply Lemma \ref{lem:applying_rigidity} we now show that finite type (separated) subschemes of $\Hom$-schemes of pure varieties are proper.
 
 \begin{proposition}\label{prop:projectivity_of_Hom}
 Let $X$ be a projective variety over $k$ which is pure over $k$. Then, for every smooth projective variety $Y$ over $k$, the locally finite type scheme $\underline{\mathrm{Hom}}_k(Y,X)$ satisfies the valuative criterion of properness over $k$. In particular, for any $P\in \QQ[t]$, the scheme $\underline{\Hom}_k^P(Y,X)$ is projective and pure over $k$.
 \end{proposition}
 \begin{proof} 
 Let $S$ be a smooth affine curve over $k$ and let $K=K(S)$ be its function. Note that $X_S :=X\times_k S$ is pure over $k$ (as $X$ and $S$ are pure over $k$).
 We claim that the injective map of sets
 \[
 \mathrm{Hom}_S(Y_S,X_S) \to \mathrm{Hom}_K(Y_K,X_K)
 \] is surjective. To do this, 
  let $f:Y_K \to X_K$ be a morphism over $K$.  Since $X$ is proper over $k$, the scheme $X_S$ is proper over $S$. In particular, by the valuative criterion of properness, there is an open $U\subset Y_S$ with $U_K\cong Y_K$ and a morphism $U\to X_S$ extending $Y_K\to X_K$ with     $\mathrm{codim} (Y_S \setminus U) \geq 2$. Since $Y_S = Y\times_k S$ is a smooth variety over $k$, the purity of $X\times_k S$ implies that the morphism $U\to X_S$ extends to a morphism $Y_S\to X_S$. This shows that $\underline{\Hom}_k(Y,X)$ satisfies the valuative criterion of properness over $k$. In particular, for any $P$ in $\QQ[t]$, the quasi-projective scheme $\underline{\Hom}_k^P(Y,X)$ over $k$ is projective over $k$.  Now, since $Z:=\underline{\Hom}_k^P(Y,X)$ is proper over $k$, by Lemma \ref{lem:applying_rigidity}, for $y$ in $Y(k)$, the evaluation morphism $\mathrm{eval}_y:Z\to X$ is finite. Thus, as $X$ is pure and $Z\to X$ is finite (hence affine), we conclude that  $Z  = \underline{\Hom}_k^P(Y,X)$ is pure (Lemma \ref{lem:hartog}). This concludes the proof.
 \end{proof}
 
 The arguments used in the proof of  Proposition \ref{prop:projectivity_of_Hom} can be used to prove the properness of other $\Hom$-schemes (and $\Hom$-stacks) as we show now. Concerning algebraic stacks, we follow the conventions of the stacks project \cite[Tag~026N]{stacks-project}.
 \begin{lemma}\label{lem:properness_of_hom} Let $X$ be a projective pure variety over $k$, and let $U\to M$ be a smooth proper representable morphism of smooth finite type separated Deligne-Mumford algebraic stacks over $k$.  Let $\mathcal{L}$ be a $M$-relative ample line bundle on $U$. Then,
the natural (representable) morphism $ \underline{\Hom}_{M}( U ,X\times M) \to  M$  of algebraic stacks satisfies the valuative criterion of properness over $k$.  Therefore, for any polynomial $P\in \QQ[t]$, the morphism $\underline{\Hom}^P_M(U,X\times M)\to M$ is proper.  
\end{lemma}
\begin{proof} Since $\underline{\Hom}^P_M(U,X\times M)\to M$ is a finite type separated morphism of finite type separated algebraic stacks over $k$, it suffices to prove the first statement. Let $S$ be a smooth affine curve over $k$,   let $S\to  M$ be a morphism, and suppose that the morphism $\Spec K(S) \to S\to  M$ lifts to a morphism $\Spec K(S) \to \underline{\Hom}_{M}(U,X\times M)$. Thus, we are given a smooth finite type morphism $U_S\to S$ of schemes and a   morphism $ U_{K(S)}\to X_{K(S)}$. By properness of (the morphism  of noetherian schemes)  $X\times_k S\to S$, there is a dense open $V\subset  U_S$ with $\mathrm{codim}(U_S\setminus V)\geq 2$ and a morphism $V\to X\times_k S$ which extends the morphism $U_{K(S)}\to X_{K(S)}$ over $K(S)$ to a morphism over $S$. Since $S$ is affine, the curve $S$ is pure over $k$ (Lemma  \ref{lem:hartog}). Therefore, as $X$ and $S$ are pure over $k$, the variety $X\times_k S$ is pure over $k$. Since $U_S$ is a smooth finite type separated Deligne-Mumford algebraic stack over $k$  and $X\times_k S$ is a pure  variety over $k$, the morphism $V\to X\times_k S$ extends to a morphism $U_S\to X\times_k S$.   This concludes the proof of the lemma. 
\end{proof}

We now combine Lemma \ref{lem:applying_rigidity} and Proposition \ref{prop:projectivity_of_Hom} to show that ``evaluation maps'' on finite type closed subschemes of $\Hom$-schemes of pure varieties are finite.

\begin{corollary}\label{cor:crit_for_qf}
 Let $X$ be a proper  pure variety over $k$ and let $Y$ be a smooth projective variety over $k$. If $Z$ is an irreducible component  of $\underline{\Hom}_k(Y,X)$, then $Z$ is projective and, for any $y$ in $Y(k)$, the evaluation morphism 
 \[
 Z\to X, f \mapsto f(y)
 \] is  finite.
\end{corollary}
\begin{proof}
Note that $Z$ is closed in $\underline{\Hom}_k(Y,X)$. Then, by Proposition \ref{prop:projectivity_of_Hom}, the quasi-projective $k$-scheme $Z$ is proper over $k$ (hence projective).  The result now follows from Lemma \ref{lem:applying_rigidity}.
\end{proof}

Let $m\geq 1$ be an integer. Let $Y$ and $X$ be proper schemes over $k$. Let $y_1,\ldots, y_m$ be elements of $Y(k)$, and let $x_1,\ldots,x_m$ be elements of $X(k)$.  The functor $$\underline{\Hom}_k((Y,y_1,\ldots,y_m), (X,x_1,\ldots,x_m))$$ parametrizing morphisms $f:Y\to X$ with $f(y_1)=x_1,\ldots, f(y_m) = x_m$ is representable by a (possibly empty) closed subscheme of $\underline{\Hom}_k(Y,X)$ which we (also) denote by $$\underline{\Hom}_k((Y,y_1,\ldots,y_m), (X,x_1,\ldots,x_m)).$$ 
 
We conclude this section with the following  proposition which  says that algebraic sets of pointed maps to a pure variety are zero-dimensional.
 
 \begin{proposition}\label{prop:zero_dimensionality}
Let $X$ be a projective pure variety over $k$. Let $m\geq 1$ be an integer, let $Y$ be a smooth projective variety over $k$, let $y_1,\ldots, y_m$ be pairwise distinct points  in $Y(k)$, and let $x_1,\ldots,x_m\in X(k)$. Then the locally finite type scheme $$\underline{\mathrm{Hom}}_k((Y,y_1,\ldots,y_m),(X,x_1,\ldots,x_m))$$ parametrizing morphisms $f:Y\to X$ with  $f(y_1) = x_1, \ldots, f(y_m) = x_m$ is zero-dimensional. 
 \end{proposition}
 \begin{proof} We show that, for $y$ in $Y(k)$ and $x$ in $X(k)$, the scheme $\underline{\Hom}_k((Y,y),(X,x))$ is zero-dimensional.  (This is clearly enough.)
To do so, let  $Z$ be an irreducible component of $\underline{\Hom}_k(Y,X)$. Note that, as $X$ is pure over $k$, it follows from Corollary \ref{cor:crit_for_qf} that, for every $y $  in $Y(k)$, the evaluation morphism $Z\to X $ of $k$-schemes given by $f\mapsto f(y)$ is finite.  This implies that  the scheme $\underline{\Hom}_k((Y,y),(X,x))$ is zero-dimensional, as required.
\end{proof}

 \section{Bounded varieties: definitions and grouplessness}\label{section:boundedness}
 By Definition \ref{def:alg_hyp}, a projective algebraically hyperbolic variety satisfies a ``strong'' form of boundedness with respect to maps from curves. As we explained in the introduction, a lot of properties of algebraically hyperbolic varieties we prove in this paper also hold for varieties satisfying (a priori) ``weaker'' properties of boundedness with respect to maps from curves. To state and prove our results, we start by defining what we mean by ``bounded'' and ``$(n,m)$-bounded'' projective schemes.

\begin{definition}\label{def:bounded} Let $n$ be a non-negative integer. A projective scheme $X$ over $k$ is \emph{n-bounded over $k$} if, for all normal  projective integral schemes $Y$ of dimension at most $n$ over $k$, the  scheme $\underline{\mathrm{Hom}}_k(Y,X)$ is of finite type over $k$.  A projective scheme $X$ over $k$ is \emph{bounded} if, for all integers $n\geq 1$, the scheme $X$ is $n$-bounded.
\end{definition}

\begin{definition}\label{def:nm_bounded}
Let $n$ and $m$ be non-negative integers. A projective scheme $X$ over $k$ is \emph{(n,m)-bounded (over $k$)} if, for all normal projective integral schemes $Y$ of dimension at most $n$ over $k$, all  pairwise distinct points $y_1,\ldots, y_m \in Y(k)$, and all $x_1,\ldots,x_m\in X(k)$, the scheme 
\[
\underline{\mathrm{Hom}}_k((Y,y_1,\ldots,y_m), (X,x_1,\ldots,x_m))
\] is of finite type.
\end{definition}

\begin{remark} \label{bounded-impications}
Note that a projective variety over $k$ is $n$-bounded over $k$ if and only if it is $(n,0)$-bounded over $k$ (by definition). Obviously, if $X$ is $n$-bounded, then $X$ is $(n-1$)-bounded. Moreover, if $X$ is $(n,m)$-bounded over $k$, then $X$ is $(n,m+1)$-bounded.
\end{remark}

\begin{proposition}\label{prop:nm_bounded_is_groupless} Let $n\geq 1$ and $m\geq 0$ be integers.
If $X$ is a projective  $(n,m)$-bounded scheme over $k$, then $X$ is groupless and pure over $k$.  
\end{proposition}
\begin{proof}  Since a projective $(a,b)$-bounded scheme is $(a,b+1)$-bounded (Remark \ref{bounded-impications}), we may and do assume that the integer $m$ is greater than $1$. Suppose that $X$ is not groupless over $k$. Let $A$ be an abelian variety  and let $f:A\to X$ be a non-constant morphism.  Let $a_1,\ldots,a_m$ be pairwise distinct points in $A[m]$.  Let $C\subset A$ be a smooth projective curve containing $a_1, \ldots, a_m$. Let $\ell > 1$ be an integer such that $\ell=1\mod m$. Then, for every $1\leq i\leq m$, as  $a_i$ is $m$-torsion, we see that $\ell a_i = a_i$ in $A$.  The morphism $f\circ [\ell]:A\to X$ sends $a_1,\ldots,a_m$ to $f(a_1),\ldots, f(a_m)$, respectively. In particular, as the morphisms $f\circ [\ell]$ correspond to $k$-points of different components of $\underline{\Hom}_k((A,a_1,\ldots,a_m), (X, f(a_1), \ldots, f(a_m)))$,  we see that $X$ is not $(1,m)$-bounded. It follows that $X$ is not $(n,m)$-bounded. 
We conclude  that a projective $(n,m)$-bounded scheme is groupless. Finally, since projective groupless varieties are pure (Proposition \ref{prop:groupless_implies_pure}), this concludes the proof.
\end{proof}

\begin{corollary}\label{cor:alg_hyp_is_groupless}
A projective algebraically hyperbolic variety over $k$ is groupless and pure over $k$.
\end{corollary}
\begin{proof}
Note that algebraically hyperbolic projective varieties are $1$-bounded. Therefore, the result follows from Proposition \ref{prop:nm_bounded_is_groupless}. 
\end{proof}

Assuming $m\geq 1$ is a positive integer, we now show that an $(n,m)$-bounded projective variety   admits only finitely many pointed maps $(Y,y_1,\ldots,y_m)\to (X,x_1,\ldots,x_m)$. The precise statement reads as follows.

\begin{lemma}\label{lem:nm_bounded_equivalence} Let $X$ be  a projective variety over $k$.
Let $n\geq 1$ and $m\geq 1$ be integers. The following are equivalent.
\begin{enumerate}
\item The projective variety $X$ is $(n,m)$-bounded over $k$. 
\item   For all   projective integral schemes $Y$ of dimension at most $n$ over $k$, all  pairwise distinct points $y_1,\ldots, y_m \in Y(k)$, and all $x_1,\ldots,x_m\in X(k)$, the set 
\[
 {\mathrm{Hom}}_k((Y,y_1,\ldots,y_m), (X,x_1,\ldots,x_m))
\] is   finite.
\end{enumerate}
\end{lemma}
\begin{proof}  
Clearly, $(2)\implies (1)$. Let us show that $(1)\implies (2)$. Replacing $Y$ by a densingularization if necessary, we may and do assume that $Y$ is smooth. Now, to show that $(1)\implies (2)$,  note that an $(n,m)$-bounded projective variety is groupless  and pure (Proposition \ref{prop:nm_bounded_is_groupless}).  Therefore, since  $X$ is a pure   projective scheme over $k$ and $Y$ is a smooth projective variety over $k$,  the scheme $\underline{\Hom}_k((Y,y_1,\ldots,y_m),(X,x_1,\ldots,x_m))$ is zero-dimensional (Proposition \ref{prop:zero_dimensionality}). By our assumption $(1)$, the latter scheme is  of finite type. As a finite type zero-dimensional $k$-scheme is finite, this concludes the proof.
\end{proof}

\begin{remark}
Note that the assumption that $m$ is positive    in Lemma \ref{lem:nm_bounded_equivalence} is necessary. Indeed, $\Hom(C,X)$ contains all  constant maps $C\to X$, and is therefore   infinite (even if $X$ is bounded). However, more interestingly, there is a bounded projective surface $X$ over $\mathbb{C}$ and a smooth projective curve $C$ over $\mathbb{C}$ such that there are infinitely many non-constant morphisms $C\to X$ (of bounded degree). Explicitly: let $C$ be a smooth projective connected curve of genus at least two  over $\mathbb{C}$, and let $X := C\times C$. Then, for every closed point $c$ of $C$, the morphism $f_c:C\to X$ defined by $f_c(d) = (d,c)$ is of bounded degree. Since $C$ is bounded, the surface $X$ is bounded. Moreover, the morphisms $f_c$ are pairwise distinct non-constant morphisms.
\end{remark}

\begin{proposition}\label{prop:1m_is_11} Let $X$ be a projective variety over $k$.
The following are equivalent.
\begin{enumerate}
\item The projective variety $X$ is $(1,1)$-bounded.
\item There is an integer $m\geq 1$ such that $X$ is $(1,m)$-bounded.
\end{enumerate}
\end{proposition}
\begin{proof}Let $m\geq 1$ be an integer and assume that $X$ is $(1,m)$-bounded. To prove the proposition,  it suffices to show that $X$ is $(1,1)$-bounded.
Let $C$ be a smooth projective curve over $k$, let $c\in C(k)$, and   let $x\in X(k)$. We now show that   set $\Hom_k((C,c),(X,x))$ is finite.

 Let  $D$ be a smooth projective connected curve and let $f:D\to C$ be a finite surjective  morphism of degree $m$ which is \'etale over $c$.   Write $\{d_1,\ldots,d_m\} = f^{-1}\{c\}$, and note that $d_1,\ldots,d_m$ are pairwise distinct points of $D$. Define $x_1 = \ldots = x_m = x$. Since $D\to C$ is surjective, the map of sets \[
 \Hom_k((C,c),(X,x)) \to  
 \Hom_k((D,d_1,\ldots,d_m),(X,x_1,\ldots,x_m)),  \quad g\mapsto g\circ f
 \]   is injective. Since $X$ is $(1,m)$-bounded, the set $ \Hom_k((D,d_1,\ldots,d_m),(X,x_1,\ldots,x_m))$ is finite (Lemma  \ref{lem:nm_bounded_equivalence}). We conclude that $\Hom_k((C,c),(X,x))$ is finite, so that $X$ is $(1,1)$-bounded, as required.
\end{proof}

We will later show that a projective variety $X$ over $k$ is $(1,1)$-bounded  over $k$ if and only if, for every $n\geq 1$ and $m\geq 1$, we have that $X$ is $(n,m)$-bounded; see Theorem \ref{thm:1m_is_nl} for a more precise statement.

\section{Hyperbolicity and boundedness along finite maps} \label{S5}

In this section we show that the notions of being algebraically hyperbolic, Kobayashi hyperbolic, and $(n,m)$-bounded (for some fixed $n$ and $m$) behave in a similar manner along finite maps.  
 
In our proofs below we will use the ``slope'' of a morphism $f:C\to X$ with respect to a fixed ample line bundle on $X$.
\begin{definition}\label{def:slope} Let $\mathcal{L}$ be an ample line bundle on a projective scheme $X$ over $k$.
Let $C\to X$ be a morphism of projective schemes over $k$ with $C$ a smooth projective connected curve over $k$. The \emph{slope $s(f)$ of $f$ (with respect to $L$}) is defined as 
\[
s(f) = \frac{\deg_C f^\ast \mathcal{L}}{\max(1,\mathrm{genus}(C))}.
\]
\end{definition}

Note that a projective scheme $X$ over $k$ is algebraically hyperbolic  over $k$ if and only if there is a real number $\alpha$ (depending only on $X$ and some fixed ample line bundle $\mathcal{L}$ on $X$) such that, for every smooth projective connected curve $C$ over $k$ and every morphism $f:C\to X$, the slope (with respect to the aforementioned fixed ample line bundle on $X$) satisfies $s(f) \leq \alpha$. On the other hand, a projective scheme $X$ over $k$ is $1$-bounded over $k$ if and only if for every smooth projective connected curve $C$ over $k$ there is a real number $\alpha_C$ (which depends only on $X$ and $C$) such that $s(f) \leq \alpha_C$.  Thus, one could say that algebraic hyperbolicity is a ``uniform'' version of $1$-boundedness, as a projective variety $X$ is algebraically hyperbolic if and only if the slope of a morphism from a smooth projective curve to $X$ is uniformly bounded.

\begin{proposition}\label{prop:finite_morphisms}
Let $f:X\to Y$ be a finite morphism of projective varieties over $k$. Then the following statements hold.
\begin{enumerate}
 \item If $Y$ is algebraically hyperbolic over $k$, then $X$ is algebraically hyperbolic over $k$.
\item Let $n\geq 1$ and $m\geq 0$ be   integers. If $Y$ is $(n,m)$-bounded over $k$, then $X$ is $(n,m)$-bounded over $k$.
\item Assume $k=\mathbb{C}$. If $Y$ is Kobayashi hyperbolic over $k$, then $X$ is Kobayashi hyperbolic over $k$.
\end{enumerate}
\end{proposition}
\begin{proof}
   To prove $(1)$, let $\mathcal{L}$ be an ample line bundle on $Y$. Since $f:X\to Y$ is finite, the line bundle  $f^\ast \mathcal{L}$ is ample on $X$. Suppose that $X$ is not algebraically hyperbolic over $k$, so that there is a smooth projective connected curve $C$ over $k$, and infinitely many   morphisms  $f_i:C\to X$ such that the slope $s(f_i) = \deg(f_i)/\mathrm{genus}(C)$ (Definition \ref{def:slope}) tends to infinity as $i$ tends to infinity, where we compute the degree of $f_i:C\to X$ with respect to $f^\ast \mathcal{L}$. For every $i$, let $g_i:= f\circ f_i$, and note that the slope of the finite morphism $g_i:C\to Y$ equals the slope of $f_i$, and is therefore unbounded. This shows that $Y$ is algebraically hyperbolic, and proves $(1)$. 
   
   To prove $(2)$, suppose that $X$ is not $(n,m)$-bounded. We now show that $Y$ is not $(n,m)$-bounded.   Let $V$ be a normal projective variety of dimension $1$,  let $v_1,\ldots,v_m$ be pairwise distinct points in $V(k)$, and let $x_1,\ldots,x_m\in X(k)$ be such that $$\underline{\Hom}_k((V,v_1,\ldots,v_m),(X,x_1,\ldots,x_m))$$ is not of finite type over $k$. Let $f_i\in \Hom((V,v_1,\ldots,v_m),(X,x_1,\ldots,x_m))$ be elements with pairwise distinct Hilbert polynomials. For $i\in \{1,\ldots,m\}$, define $y_i:= f(x_i)$ and $g_i:= f_i\circ f$. Note that the elements $$g_i\in \Hom((V,v_1,\ldots,v_m),(Y,y_1,\ldots,y_m)) $$ have pairwise distinct Hilbert polynomial.   This shows that $$\underline{\Hom}_k((V,v_1,\ldots,v_m), (Y,y_1\ldots,y_m))$$ is not of finite type, so that $Y$ is not $(n,m)$-bounded over $k$.
   
  We note that $(3)$ is due to Kwack \cite[Theorem~1]{Kwack}. (One could also use Brody's lemma and the analogous statement for Brody hyperbolicity. One could also appeal to \cite[Proposition~3.2.11]{Kobayashi}) This concludes the proof.
\end{proof}

\begin{corollary}\label{cor:projectivity_of_Homs} Let $X$ be a projective   scheme over $k$. Let $Y$ be a normal  projective scheme and let $P\in \QQ[t]$ be a non-zero polynomial. Then  the following statements hold.
 \begin{enumerate}
 \item If $X$ is algebraically hyperbolic over $k$, then $\underline{\Hom}^P_k(Y,X)$  is a projective algebraically hyperbolic scheme over $k$ with $\dim\underline{\Hom}^P_k(Y,X) \leq \dim X$.
 \item If $n\geq 1$ and $m\geq 0$ are   integers and $X$ is $(n,m)$-bounded over $k$, then $\underline{\Hom}^P_k(Y,X)$  is a projective $(n,m)$-bounded scheme over $k$  with $\dim\underline{\Hom}^P_k(Y,X) \leq \dim X$.
 \end{enumerate}
\end{corollary}
\begin{proof} Let $X$ be as in $(1)$ or $(2)$. Note that, replacing $Y$ by a desingularization if necessary, we may and do assume that $Y$ is smooth.
First, note that $X$ is groupless by Corollary \ref{cor:alg_hyp_is_groupless}) and Proposition \ref{prop:nm_bounded_is_groupless}. Therefore, as $X$ is proper and groupless, it follows that $X$ is  pure  (Proposition \ref{prop:groupless_implies_pure}). Thus,  since $Z:= \underline{\Hom}^P_k(Y,X)$ is a finite type closed subscheme of $\underline{\Hom}_k(Y,X)$, by Corollary \ref{cor:crit_for_qf},  the scheme $ \underline{\Hom}^P_k(Y,X) $ is projective and, for every $y$ in $Y(k)$, the  evaluation morphism $ \mathrm{eval}_y\colon \underline{\Hom}^P_k(Y,X)\to X$ is finite. This implies that $\dim \underline{\Hom}_k^P(Y,X)\leq \dim X$. Now, if $X$ is algebraically hyperbolic (resp. $(n,m)$-bounded), it follows from  Proposition \ref{prop:finite_morphisms} that $\underline{\Hom}_k^P(Y,X)$ is algebraically hyperbolic (resp. $(n,m)$-bounded). This concludes the proof.
\end{proof}

\begin{proposition}
Let $f:X\to Y$ be a finite \'etale morphism of projective varieties. Then the following statements hold.
\begin{enumerate}
\item If $X$ is algebraically hyperbolic over $k$, then $Y$ is algebraically hyperbolic over $k$.
\item If $n\geq 1$ and $m\geq 0$ are integers and $X$ is $(n,m)$-bounded over $k$, then $Y$ is $(n,m)$-bounded over $k$.
\item Assume $k=\mathbb{C}$. If $X^{\an}$ is Kobayashi hyperbolic, then $Y^{\an}$ is Kobayashi hyperbolic.
\end{enumerate}
\end{proposition}
\begin{proof} Let $d =\deg(Y/X)$, let $\mathcal{L}$ be an ample line bundle on $Y$, and note that $f^\ast \mathcal{L}$ is ample on $X$.

To prove $(1)$, assume that $X$ is algebraically hyperbolic over $k$, and let $\alpha$ be a real number (which depends on $\mathcal{L}$ and $f:X\to Y$) such that, for every smooth projective connected curve $C'$ over $k$ and every morphism $f':C'\to X$ we have $s(f') \leq \alpha$. To show that $Y$ is algebraically hyperbolic over $k$, let $C$ be a smooth projective curve over $k$ and let $f:C\to Y$ be a morphism. Let $D:= C\times_Y X$, and let $g:D\to X$ be the natural morphism. Note that $D$ is a smooth projective curve over $k$. We now bound the slope $s(f)$ of $f$ (Definition \ref{def:slope}). Note that $\mathrm{genus}(D) = d \mathrm{genus}(C) > 0$ and that
$$ 
 \alpha \geq \quad s(g) = \frac{\deg_D g^\ast f^\ast \mathcal{L}}{\mathrm{genus}(D)} = \frac{ d \deg_C f^\ast \mathcal{L}}{d \mathrm{genus}(C)} = s(f).
$$ In particular, the slope of $f$ is bounded by $\alpha$. We conclude that $Y$ is algebraically hyperbolic over $k$.

To prove $(2)$,  
    assume that $Y$ is not $(n,m)$-bounded. Let $V$ be a normal projective variety of dimension at most $n$ over $k$, let $v_1,\ldots,v_m$ be points in $V(k)$, let $y_1,\ldots, y_m$ be points in $Y(k)$, and let $f_i:V\to Y$ be a sequence of morphisms with pairwise distinct Hilbert polynomials and $f(v_i) = y_i$. 
   Since $k$ is an algebraically closed field of characteristic zero, it follows from \cite[Expos\'e~II. Theorem 2.3.1]{SGA7I} that the set of $k$-isomorphism classes of (normal) projective varieties $W$ such that there is a  finite \'etale morphism $W\to V$ of degree at most $d$ is finite. Therefore, replacing $(f_i)_{i=1}^\infty$ by a subsequence if necessary, we may and do assume that, for all positive integers $i$, we have $W:=V\times_{f_1,Y,f} X \cong V \times_{f_i,Y,f} X$. Let $g_i:W = V_i\times_Y X\to X$ be the natural morphism. For every positive integer $i$, consider the morphism $W= V_i\times_Y X\to V$ and let $w_i$ be a point lying over $v_i$. Replacing $(f_i)_{i=1}^\infty$ by  a subsequence if necessary,  we may and do assume that $g_i(w_1), \ldots,  g_i(w_m)$ are independent of $i$. Let $x_1 := g_1(w_1), \ldots, x_m = g_1(w_m)$. Note that $g_i$ is an element of 
   \[
   \Hom_k((W,w_1,\ldots,w_m), (X,x_1,\ldots,x_m)).
   \] Since the $f_i$ have pairwise distinct Hilbert polynomial, it follows that the $g_i$ have pairwise distinct Hilbert polynomial. This shows that $X$ is not $(n,m)$-bounded.  

  Note that  $(3)$ follows from \cite[Theorem~3.2.8.(2)]{Kobayashi}. (One can also use Brody's lemma and   the easier to establish analogous statement of $(3)$ for Brody hyperbolicity to prove $(3)$.)
\end{proof}

\section{Finiteness results for bounded varieties} \label{S6}
In this section we prove finiteness results for certain moduli spaces of maps. The main ingredients in this section are the theorem of Hwang--Kebekus--Peternell (Theorem \ref{thm:HKP}), and the properties of Hom-schemes of pure varieties   established in Section \ref{S3}.

Our first lemma gives the finiteness of surjective maps from a given projective variety $Y$ to a bounded projective variety $X$.

\begin{lemma}\label{lem:boundedness_implies_fin_sur} 
Let $n\geq 1$ be an integer. Let $X$ be a projective $n$-bounded variety over $k$. If $Y$ is a (reduced) projective variety of dimension at most $n$, then the set of surjective morphisms $Y\to X$ is finite.
\end{lemma}
\begin{proof} 
We may and do assume that $Y$ is integral.
Let $Y'\to Y$ be the normalization of $Y$. Note that the natural map of sets $\mathrm{Sur}_k(Y,X) \to \mathrm{Sur}_k(Y',X)$ is injective. Therefore, replacing $Y$ by its normalization if necessary, we may and do assume that $Y$ is normal. Since $X$ is $n$-bounded, it follows that $X$ is groupless (Proposition \ref{prop:nm_bounded_is_groupless}). Therefore, 
by Hwang-Kebekus-Peternell's theorem (Theorem \ref{thm:HKP}), the scheme $\underline{\mathrm{Sur}}_k(Y,X)$ is zero-dimensional. Since $X$ is $n$-bounded and $Y$ is a normal projective variety, the scheme   $\underline{\Hom}_k(Y,X)$ is of finite type over $k$. In particular, the zero-dimensional scheme $\underline{\mathrm{Sur}}_k(Y,X)$ is of finite type over $k$, and is therefore finite over $k$. This concludes the proof.
\end{proof}

\begin{corollary}\label{Aut-fin-bound}   Let $n\geq 1$ be an integer. Let $X$ be an $n$-dimensional (reduced) projective $n$-bounded variety over $k$.
Then  $X$ has only finitely many surjective endomorphisms, and every surjective endomorphism of $X$ is an automorphism of $X$. In particular, $\Aut(X)$ is finite.  
\end{corollary}
\begin{proof} 
This follows from Lemma \ref{lem:boundedness_implies_fin_sur}.
\end{proof}

 \begin{proposition} \label{dom-fin-norm}  Let $n\geq 1$ be an integer. Let $X$ be a projective $n$-bounded scheme over $k$. 
If  $Y$ is a (reduced) projective variety of dimension at most $n$ over $k$, then the set of dominant rational maps $Y\dashrightarrow X$ is finite.  
 \end{proposition}
 \begin{proof}   Replacing $Y$ by a desingularization if necessary, we may and do assume that $Y$ is a smooth projective variety over $k$.
Since $X$ is $n$-bounded, it is pure (Proposition \ref{prop:nm_bounded_is_groupless}). Therefore, every dominant rational map $Y\dashrightarrow X$ extends uniquely to a well-defined surjective morphism $Y\to X$ (Lemma \ref{lem:rat_is_mor}). The result now follows from Lemma \ref{lem:boundedness_implies_fin_sur}.
 \end{proof}
 
 \begin{corollary} \label{Hom-subsch-finite}  Let $n\geq 1$ be an integer. 
 Let $X$ be a projective $n$-bounded scheme over $k$. Let $Y$ be a projective  scheme of dimension at most $n$ over $k$. Let $A\subset X$ be a non-empty reduced closed subscheme of $X$, and let $B\subset Y$ be a non-empty reduced closed subscheme of  $Y$. Then the set \[
 \{f\in \mathrm{Hom}_k(Y,X) \ | \ f(B) = A\}
 \] is finite.
 \end{corollary}
 \begin{proof} Note that the inclusion $A\to X$ is finite. Therefore, 
since $X$ is $n$-bounded over $k$, it follows that $A$ is $n$-bounded over $k$ (Proposition \ref{prop:finite_morphisms}). Thus, as $\dim B \leq \dim Y \leq n$ and $B$ is reduced, the set $\mathrm{Sur}(B,A)$ of surjective morphisms $B\to A$ is finite (Lemma \ref{lem:boundedness_implies_fin_sur}).  Fix $b$ in $B(k)$. Then,   the finiteness of $\mathrm{Sur}(B,A)$ implies that the set $$I:=I_b:= \{a\in A \ | \ \textrm{there is a surjective morphism } f:B\to A \textrm{ with } f(b) = a\}$$ is finite.  Now,  if $\Hom_k((Y,B),(X,A))$ denotes the set of morphisms $f:Y\to X$ with $f(B)=A$, then it is clear that 
$$\Hom_k((Y,B),(X,A))\subset \cup_{a\in I} \Hom_k((Y,b), (X,a)).$$ Thus, as the set $I$ is finite, it suffices to show that, for every $a\in X(k)$, the set $$\Hom_k((Y,b),(X,a))$$ is finite. Since an $n$-bounded variety is $(n,1)$-bounded (Remark \ref{bounded-impications}), the latter finiteness follows from  Lemma \ref{lem:nm_bounded_equivalence}.
 \end{proof}

\section{Geometricity theorems} \label{S7}
Note that a variety over $k$  of general type  remains of general type   after any algebraically closed field extension. In other words, the property that a variety is of general type is ``geometric'' (in the sense that it persists over any algebraically closed field extension). Similarly,  by Lemma \ref{lem:ge}, the property that a variety is groupless is also ``geometric'', and the property that a variety is pure is also ``geometric'' (Remark \ref{remark:pure}). Therefore, as the
  Green--Griffiths--Lang's conjecture   says that a projective variety is groupless if and only if it is algebraically hyperbolic, we see that the Green--Griffiths--Lang conjecture predicts that algebraic hyperbolicity is a  ``geometric'' property (i.e., persists over any algebraically closed field extension). We now prove this. 

\begin{theorem}[Algebraic hyperbolicity is a geometric property]\label{thm1}
Let $X$ be a projective scheme over $k$ and let $k \subset L$ be an extension of algebraically closed fields of characteristic zero. 
Then $X$ is algebraically hyperbolic over $k$ if and only if $X_L$ is algebraically hyperbolic over $L$.
\end{theorem}
 
\begin{proof}  Since $X$ is algebraically hyperbolic over $k$, it is groupless over $k$ (Corollary \ref{cor:alg_hyp_is_groupless}).    In particular, $X_L$ is groupless over $L$ (by Lemma \ref{lem:gr}). In particular, the variety $X_L$  admits no maps from a smooth projective curve of genus at most one. 

Let $\alpha$ be a real number such that, for any $g\geq 2$, any $C\in \mathcal{M}_g(k)$, and any non-constant morphism $f:C\to X$, we have that the slope $s(f)$, as  defined in Definition \ref{def:slope},  satisfies $s(f) \leq \alpha$. Such a real number $\alpha$ exists, as $X$ is algebraically hyperbolic over $k$.  

    Let $C$ be a smooth projective curve  of genus $g$ (at least two) over $L$ and let $g:C\to X_L$ be a non-constant morphism.   Choose a smooth affine variety $U$ over $\Spec k$, choose a smooth proper geometrically connected  genus $g$ curve $\mathcal{C}\to U$ over $U$ with $\mathcal{C}_{L}\cong C$, and choose a morphism  $\mathcal{C} \to X\times U$ of $U$-schemes which equals $C\to X_L$ after pull-back along $\Spec L\to U$.  Let $u\in U(k)$  and consider the induced morphism $f:\mathcal{C}_{u} \to X\times \{u\} \cong X$. Note that the slope of the morphism $f:\mathcal{C}_u\to X$ equals the slope of the morphism $g:C\to X_L$, i.e., $s(g) = s(f)$.  Therefore, since $\mathcal{C}_{u}$ is in $\mathcal{M}_g(k)$ and $\mathcal{C}_{u} \to X$ is non-constant, we   have that $$s(g) = s(f) \leq \alpha.$$ This implies that $X_L$ is algebraically hyperbolic over $L$, and concludes the proof.  
\end{proof}

Motivated by Green--Griffiths--Lang conjecture, and the     similarities between the notions of boundedness and algebraic hyperbolicity established in Sections \ref{section:boundedness} and \ref{S5}, we now establish the geometricity of boundedness. 

To do so, for $g\geq 2$ an integer, let $\mathcal{M}_g$ be the stack of smooth proper curves of genus $g$ over $\ZZ$, and let $\mathcal{U}_g\to \mathcal{M}_g$ be the universal smooth proper geometrically connected curve of genus $g$ over $\mathcal{M}_g$. Recall that $\mathcal{M}_g$ is a smooth finite type separated Deligne-Mumford algebraic stack over $\ZZ$. More generally, for $g\geq 2$ and $m\geq 1$ an integer, let $\mathcal{M}_{g,m}$ be the stack of $m$-pointed smooth proper geometrically connected curves of genus $g$, and let $\mathcal{U}_{g,m}\to \mathcal{M}_{g,m}$ be the universal family.

\begin{theorem} [$1$-boundedness is a geometric property]\label{thm:1b_is_geom}
Let $X$ be a projective scheme over $k$. Then $X$ is $1$-bounded over $k$ if and only if $X_L$ is $1$-bounded over $L$.
\end{theorem}
\begin{proof} Clearly, if $X_L$ is $1$-bounded over $L$, then $X$ is $1$-bounded over $k$. To prove the converse, 
assume that $X_L$ is not $1$-bounded over $L$. Then, there is an integer $g\geq 2$, a smooth projective curve $C$ over $L$ of genus $g$, and an increasing sequence of integers $d_1< d_2 <\ldots$ such that $\underline{\Hom}^{d_i}_L(C, X_L)$ has an $L$-point. 

Since $X$ is $1$-bounded over $k$, it follows that $X$ is pure and groupless (Proposition \ref{prop:nm_bounded_is_groupless}). Define $M:= \mathcal{M}_{g}\otimes_{\ZZ} k$ and $U:= \mathcal{U}_{g}\otimes_{\ZZ} k$. Note that $U\to M$ is a smooth proper representable morphism of smooth finite type separated Deligne-Mumford algebraic stacks over $k$. Moreover, the relative dualizing sheaf $\omega_{U/M}$ is an $M$-relative ample line bundle on $U$ (as $g\geq 2$).

Note that, as $X$ is projective and pure over $k$, the natural  morphism  $\underline{\Hom}_{M}(U, X\times M) \to  M $  satisfies the valuative criterion of properness over $k$ (Lemma \ref{lem:properness_of_hom}). In particular, for any integer $d$, the finite type separated morphism $\phi_d:\underline{\Hom}^d_{M}(U,X\times M) \to  M $ is proper.    Let $Z_d\subset M$ be the stack-theoretic  image of $\phi_d$, and note that $Z_d$ is a closed substack of $M$.

Now, since $\underline{\Hom}^{d_i}_L(C,X_L)$ has an $L$-point  for all $i\in \{1, 2,\ldots\}$, the algebraic stack $Z_{d_i}$ (over $k$) has an $L$-object (corresponding to the curve $C$) for all $i\in \{1,2,\ldots\}$. Define $Z:=\cap_{i=1}^\infty Z_{d_i}$, and note that $Z$ is a closed substack of $\mathcal{M}_g$ with an $L$-point. Since $Z$ is a finite type separated algebraic stack over $k$ with an $L$-point, we conclude that $Z(k)$ is non-empty. This means that $Z_{d_i}(k) \neq \emptyset$ for all $i = 1, 2, \ldots$. Thus, there is a smooth projective curve $C'$ of genus $g$ and a sequence of morphisms $g_i:C'\to X$ of increasing degree. This shows that $X$ is not $1$-bounded over $k$, and concludes the proof.
\end{proof}

The argument to prove Theorem \ref{thm:1b_is_geom} can be used to show that $(1,m)$-boundedness is a geometric property, as we show now. 
  \begin{theorem}[$(1,m)$-boundedness is a geometric property]\label{thm:1m_is_geom} Let $X$ be a projective scheme over $k$, and let $m\geq 1$.
  Then $X$ is $(1,m)$-bounded over $k$ if and only if $X_L$ is $(1,m)$-bounded over $L$.
  \end{theorem}
  \begin{proof} We follow the proof of Theorem \ref{thm:1b_is_geom} with only minor modifications.
  Assume that $X_L$ is not $(1,m)$-bounded over $L$. Then, there is an integer $g\geq 2$,  a smooth proper   connected curve $C$ over $L$ of genus $g$, pairwise distinct    points $c_1,\ldots,c_m\in C(L)$, and  points $x_1,\ldots,x_m \in X(L)$ such that  the scheme $$\underline{\Hom}_L((C,c_1,\ldots,c_m), (X_L,x_1,\ldots,x_m))$$ is not of finite type over $L$. We fix $C$, $c_1,\ldots, c_m$, and $x_1,\ldots, x_m$ with this property.
 
  Define $M:= \mathcal{M}_{g,m} \otimes_{\ZZ} k$ and $U:= \mathcal{U}_{g,m}\otimes_{\ZZ} k$.  Note that $U\to M$ is a smooth proper   morphism of smooth finite type separated Deligne-Mumford algebraic stacks over $k$ which is representable by schemes. Moreover, there is an $M$-relative ample line bundle on $U$.
   
  For any integer $d$, let $$\phi_d:\underline{\Hom}_{M}^d( U, X\times M) \to  M \times X^m$$ be the morphism defined by  $$((D,d_1,\ldots,d_m), f:D\to X) \mapsto ((D, d_1,\ldots,d_m), (f(d_1), \ldots, f(d_m)). $$   
  Since $X$ is $(1,m)$-bounded over $k$, it follows that $X$ is pure  (Proposition \ref{prop:nm_bounded_is_groupless}).
Therefore, as $X$ is projective and pure over $k$, the natural  morphism    $$\underline{\Hom}_{M}^d( U, X\times M) \to  M $$  is proper (Lemma \ref{lem:properness_of_hom}). 
  As $X^m$ is separated over $k$ and the composed morphism  $\underline{\Hom}_{M}^d( U, X\times M) \to  M \times X^m\to M$ is proper, the morphism $$\phi_d:\underline{\Hom}_{M}^d( U, X\times M) \to  M \times X^m$$   is proper \cite{Liu2}. 
  
  Let $Z_d$ be the image of $\phi_d$ in $M\times X^m$, and note that  $Z_d$ is a closed substack of $M\times X^m$. Since $\underline{\Hom}_L((C,c_1,\ldots,c_m), (X_L,x_1,\ldots,x_m))$ is not of finite type over $L$ (by assumption), there is a sequence of integers $d_1< d_2 < \ldots$ such that $$ \Hom_L((C,c_1,\ldots,c_m), (X_L,x_1,\ldots,x_m))=\underline{\Hom}_L((C,c_1,\ldots,c_m), (X_L,x_1,\ldots,x_m))(L)$$ is non-empty.  In particular, $Z_{d_i}(L)$ is non-empty. Define  $Z:= \cap_{i=1}^\infty Z_{d_i}$. Then $Z$ is a closed substack of $M\times X^m$ (over $k$) with an $L$-point. Thus, it follows that $Z(k) \neq \emptyset$.    This means precisely that there is a smooth projective connected curve $C'$ of genus $g$ over $k$ , pairwise distinct points $c_1', \ldots, c_m'\in C(k)$, and points $x_1',\ldots,x_m' \in X(k)$ such that $\underline{\Hom}_k((C', c_1', \ldots, c_m'), (X,x_1',\ldots,x_m'))$ is not of finite type over $k$. This shows that $X$ is not $(1,m)$-bounded over $k$, and concludes the proof of the theorem.  
  \end{proof}

 \section{Relating $(1,m)$-boundedness and $(n,m)$-boundedness} \label{S8}
 
 We use the geometricity theorems in the previous section, and a specialization argument, to prove that $(1,1)$-bounded varieties are $(n,1)$-bounded for every $n\geq 1$. To prove our result we use the following application of Bertini's theorem.

   \begin{lemma}\label{lem:bertini}  
  Let $X$ be a variety over an uncountable algebraically closed field $k$. Let $I$ be a countable set and let $(Z_i)_{i\in I}$ be a collection of proper closed subsets of $X$. Let $S\subset \cap_{i\in I} Z_i$ be a finite (possibly empty) closed subset. Then there is a smooth irreducible curve $C\subset X$ such that, for all $i$ in $I$, the set $C\cap Z_i$ is finite and contains $S$.
\end{lemma}
\begin{proof} 
Since $k$ is uncountable and $Z_i\neq X$ for all $i$ in $I$, we have that $X(k) \neq \cup_{i\in I} Z_i(k)$.   Therefore, there is a $k$-point $Q$  in $X$  not contained in any of the $Z_i$. By Bertini's theorem, a general complete intersection curve $C$ containing  the set $S$ and the point $Q$ is smooth and  irreducible. For every $i$ in $I$,   the intersection $Z_i \cap C$ does not contain the specified point $Q$. Therefore, the intersection is a proper closed subset of the irreducible curve $C$. We conclude that, for all $i$ in $I$, the intersection of $C$ and $Z_i$ is a finite subset of $C$ containing $S$.
\end{proof}
 \begin{proposition}  \label{prop:1m_is_nm} Let $m\geq 1$ be an integer. 
 Let $X$ be a $(1,m)$-bounded projective variety. Then, for every integer $n\geq 1$, the projective variety $X$ is $(n,m)$-bounded.  
 \end{proposition}
 \begin{proof}
By the geometricity of $(1,m)$-boundedness (Theorem \ref{thm:1m_is_geom}), for any algebraically closed field extension $k\subset L$, the projective scheme $X_L$ is $(1,m)$-bounded over $L$ (Theorem \ref{thm:1b_is_geom}). Therefore, to prove that $X$ is $(n,m)$-bounded, we may and do assume that $k$ is uncountable.

Now,  assume that $X$ is not $(n,m)$-bounded over $k$. Let $Y$ be a projective variety of dimension at most $n$ over $k$, let $y_1,\ldots,y_m\in Y(k)$ be pairwise distinct points, let $x_1,\ldots,x_m\in X(k)$, and let $f_1, f_2, f_3, \ldots$ be pairwise distinct morphisms $Y\to X$ such that $$f_i\in \Hom_k((Y,y_1,\ldots,y_m), (X,x_1,\ldots,x_m)).$$ (We will show that this leads to a contradiction.) 

For any pair of positive integers, define $Y^{n,m} := \{y \in Y\ | \ f_n(y) = f_m(y)\}$. Note that $Y^{n,m}$ is a proper closed subset of $Y$ which contains the points $y_1,\ldots,y_m$. (The fact that $Y^{n,m}\neq Y$ is equivalent to the fact that $f_n\neq f_m$.)  

 Let $I = \ZZ_{\geq 1}\times \ZZ_{\geq 1}\setminus \Delta$ be the set of pairs of   distinct positive integers.  For $i$  in $I$ (corresponding to $(n,m)$), define $Z_i:= Y^{n,m}$. As the collection of proper closed subsets $(Z_i)_{i\in I}$ is countable and contains $\{y_1,\ldots,y_m\}$, it follows from Lemma \ref{lem:bertini} that there is a smooth projective connected curve $C$ in $Y$ such that the intersection of $C$ with any $Z_i$ is finite and contains $\{y_1,\ldots,y_m\}$. This means that the morphisms $f_i$ restricted to $C$ are all still pairwise distinct. Thus, their restrictions $f_i|_C$ give rise to  pairwise distinct elements of $\Hom_k((C,y_1,\ldots,y_m), (X,x_1,\ldots,x_m))$. This implies that $$\Hom_k((C,y_1,\ldots,y_m), (X,x_1,\ldots,x_m))$$ is infinite. By Lemma \ref{lem:nm_bounded_equivalence}, we see that $\underline{\Hom}_k((C,c_1,\ldots,c_m),(X,x_1,\ldots,x_m))$ is not of finite type. In particular, $X$ is not $(1,m)$-bounded over $k$. This contradicts our hypothesis.
 \end{proof}
 
 \begin{corollary}\label{cor:1m_is_n1}
 Let $X$ be a projective variety over $k$. Assume that there is an integer $m\geq 1$ such that $X$ is $(1,m)$-bounded. Then, for every $n\geq 1$, the projective variety   $X$ is $(n,1)$-bounded over $k$.
 \end{corollary}
 \begin{proof}
 Since $X$ is $(1,m)$-bounded, it is $(1,1)$-bounded (Proposition \ref{prop:1m_is_11}). Therefore, it is $(n,1)$-bounded (Proposition \ref{prop:1m_is_nm}).
 \end{proof}

 \begin{theorem}\label{thm:1m_is_nl}
 Let $X$ be a projective variety over $k$. Then the following are equivalent.
 \begin{enumerate}
 \item There exist $n\geq 1$ and $m\geq 1$ such that $X$ is $(n,m)$-bounded.  
 \item For every $n\geq 1$ and $m\geq 1$, we have that $X$ is $(n,m)$-bounded over $k$. 
 \end{enumerate} 
 \end{theorem}
 \begin{proof}
This follows from Corollary \ref{cor:1m_is_n1}. 
 \end{proof}

 \section{Relating $1$-boundedness, boundedness, and  algebraic hyperbolicity} \label{S9}
 The property of being bounded has to be (by definition) ``tested'' on maps from all projective varieties. In this section, we prove that a $1$-bounded variety is in fact bounded, i.e., one can ``test'' boundedness of a variety on maps from curves. This result is an algebraic analogue of the complex-analytic fact that one can ``test'' the \emph{Borel hyperbolicity} of a variety on holomorphic maps from a curve \cite[Theorem.~1.5]{JKuch}.  To prove the main result of this section, we start with a simple intersection-theoretic lemma.
 \begin{lemma} \label{Claim}
 Let $D$ be a very ample divisor on a reduced projective scheme $Y$ over $k$ of dimension at least two.  Let $\kappa$ be a positive real number. Then, the set of numerical equivalence classes of big base-point free divisors $L$ with intersection number $L \cdot D^{\dim Y -1} \leq \kappa$ is finite. 
 \end{lemma}
 \begin{proof} (See Remark \ref{referee} below for an alternative proof.) 
 Let $f:\widetilde{Y}\to Y$ be a projective birational surjective morphism with $\widetilde{Y}$ a smooth projective variety over $k$.  By the projection formula, we have that $(f^\ast L) \cdot (f^\ast D)^{\dim \widetilde{Y}-1} = L\cdot D^{\dim Y-1} \leq \kappa$. Therefore, since $f^\ast:\NS(Y)\to \NS(\widetilde{Y})$ is injective, replacing $Y$ by $\widetilde{Y}$ if necessary, we may and do     assume that $Y$ is a  smooth projective    variety over $k$. Moreover, we may and do assume that $Y$ is connected.

 Suppose that  $\dim Y=2$. Define $e:= L \cdot L$, $g:=L \cdot D $, and $h:=D \cdot D $. Then $(hL-gD) \cdot D=0$. Also,  by the Hodge index theorem, the inequality $(hL-gD)^2 \leq 0$ holds. Therefore,  since $h$ is fixed and $g\leq \kappa$, we conclude that  $$L \cdot L = e = \frac{g^2}{h} \leq \frac{\kappa^2}{h}$$ is bounded from above. Now, since $L$ is big and base-point free, the general member $C$ of the linear system defined by $L$ is a smooth projective connected curve. Let $\mathrm{genus}(L)$ be the genus of $C$. Since $2\mathrm{genus}(L)-2=L\cdot(K_Y+L)$, we see that  $2\mathrm{genus}(L)-2$ is also bounded.  Thus, the lemma holds when $\dim Y\leq 2$.  

 Therefore, to prove the lemma,  we may and do assume that $\dim Y\geq 3$. In this case, by  the Lefschetz hyperplane theorem, the induced map from the Picard group of $Y$ to the Picard group of a smooth hyperplane section is injective.  Therefore, the lemma follows from induction on $\dim Y$. 
 \end{proof}
 
 \begin{remark}\label{referee}
The refere has pointed out to us that one can show that, with the notation of Lemma \ref{Claim},  the set of numerical equivalence classes of \emph{big}   divisors $L$ with intersection number $L \cdot D^{\dim Y -1} \leq \kappa$ is finite. Indeed, as $\mathrm{char}(k)=0$, 
the linear form $L\mapsto L\cdot D^{\dim Y -1}$  defines a half-space in the N\'eron-Severi group of $Y$ whose intersection with the big cone is a bounded convex subset, hence finite. We  have chosen to leave Lemma \ref{Claim} as stated, because its proof is also valid in positive characteristic (replacing the resolution of singularities $\widetilde{Y}\to Y$ by an alteration).
 \end{remark}
 
 \begin{theorem}[1-bounded implies bounded] \label{thm:1_implies_n}
 Let $X$ be a $1$-bounded projective scheme over $k$. Then $X$ is bounded over $k$.
 \end{theorem} 
 \begin{proof} 
We show by induction  on $n\geq 1$ that $X$ is $n$-bounded over $k$.   By assumption, the projective scheme $X$ is $1$-bounded. Thus,  let $n  \geq 2$ and assume that    $X$ is $(n-1)$-bounded. 
 Note that, for $Y$ a projective normal variety over $k$,  the Hilbert polynomial of a morphism $f: Y \to X$ is uniquely determined by the numerical equivalence class of $f^* {\mathcal O} (1)$. Indeed, by Hirzebruch-Riemann-Roch, $\chi(f^*({\mathcal O} (d)) = \deg \mathrm{ch}(f^*({\mathcal O} (d)) \cdot \tau_Y$, where $\tau_Y$ is the refined Todd class (as in Fulton \cite{Fulton}), and the Chern character depends only on the first Chern class, which is determined by the numerical equivalence class of $f^* {\mathcal O} (1)$. 
 
 Assume  that  $X$ is not $n$-bounded. Then, there is a projective  normal  $n$-dimensional variety $Y$ over $k$ and morphisms $f_1, f_2, f_3,\ldots$ from $Y$ to $X$    with pairwise distinct Hilbert polynomials.
Note that the numerical equivalence classes of $f_1^* {\mathcal O} (1), f_2^\ast \mathcal{O}(1), \ldots$ are pairwise distinct. Now, to arrive at a contradiction, fix a very ample divisor class $D$ on $Y$. From Lemma \ref{Claim} it follows that  $$f_i^*({\mathcal O} (1)) \cdot D^{\dim Y -1} \to \infty, \quad i\to \infty.$$  In particular, we have that $f_i^*({\mathcal O} (1))|_D \cdot D|_D^{\dim Y -2} \to \infty$.  Since $D$ is a  smooth projective variety with $\dim D = n-1 <n$,   this contradicts the fact that $X$ is $(n-1)$-bounded. We conclude that $X$ is $n$-bounded, as required. 
 \end{proof}
 As an application of our results, we obtain that $(n,m)$-boundedness (and thus boundedness) is a geometric property. 
 \begin{corollary}[Boundedness is a geometric property] Let $n\geq 1$ and $m\geq 0$ be integers.
 Let $k\subset L$ be an extension of algebraically closed fields of characteristic zero. A projective variety $X$ over $k$ is $(n,m)$-bounded over $k$ if and only if $X_L$ is $(n,m)$-bounded over $L$.
 \end{corollary}
 \begin{proof}
 If $X$ is $(n,m)$-bounded, then $X$ is $(1,m)$-bounded. In particular, if $m=0$, then $X_L$ is $(1,m)$-bounded by Theorem \ref{thm:1b_is_geom}. Moreover, if $m\geq 1$, then $X_L$ is $(1,m)$-bounded by  Theorem \ref{thm:1m_is_geom}. Now, as $X_L$ is $(1,m)$-bounded, for $m=0$, it follows from Theorem \ref{thm:1m_is_nl} that $X_L$ is $(n,m)$-bounded. If $m\geq 1$, then  Theorem \ref{thm:1_implies_n} implies that $X_L$ is $(n,m)$-bounded. This concludes the proof.
 \end{proof}

The relation between algebraic hyperbolicity and bounded varieties is provided by the following theorem. 
\begin{theorem}\label{alg-hyp-bounded}
A projective algebraically hyperbolic scheme over $k$ is bounded over $k$.
\end{theorem}
\begin{proof}  (This follows from 
\cite[Theorem~1.7]{KovacsLieblichErratum}. We give a self-contained proof using the results of this paper.) Let $X$ be a projective algebraically hyperbolic variety over $k$. Then, for every projective normal (hence smooth) curve $C$ over $k$, the degree of any morphism $C\to X$ is bounded linearly in the genus of $C$. In particular, the scheme $\underline{\Hom}_k(C,X)$ is of finite type over $k$.  This implies (clearly) that $X$ is $1$-bounded over $k$. Therefore, $X$ is bounded over $k$  (Theorem \ref{thm:1_implies_n}).
\end{proof}

 It seems reasonable to suspect that $(n,m)$-bounded varieties are in fact bounded. Indeed, as we explain in Section \ref{section:conjectures}, 
the Green--Griffiths--Lang conjecture implies that a $(1,m)$-bounded projective variety is $1$-bounded, and hence bounded (Theorem \ref{thm:1_implies_n}). 
 In the direction of this ``reasonable'' expectation, we prove the following result.  
 
 \begin{proposition} Let   $m \geq 1$ be an integer, and let $X$ be a  $(1,m)$-bounded  projective scheme over $k$. Let $Y$ be a projective variety over $k$.    Then, almost all (non-empty) connected components of  the scheme $\underline{\Hom}_k(Y,X)$  have dimension $< \dim X$. 
\end{proposition}
\begin{proof} Note that $X$ is $(n,1)$-bounded for all $n\geq 1$ (Corollary \ref{cor:1m_is_n1}). 
Let $y \in Y(k)$. Consider the evaluation map $\mathrm{eval}_y: \underline{\Hom}_k(Y,X) \to X$ defined by $f\mapsto f(y)$. Suppose that $\underline{\Hom}_k(Y,X)$ has infinitely many pairwise distinct connected components $H_1, \ldots$ of dimension  at least $ \dim X$. Then, as the restriction $\mathrm{eval}_y:H_i\to X$ of $\mathrm{eval}_y$ to $H_i$ is finite, it is surjective. Let $x$ be any point in $X$. Then, for every $i$,  the fibre of $H_i\to X$ over $x$ is non-empty. Thus, for every $i$, the set $\Hom((Y,y),(X,x))$ contains a point from $H_i$, and is therefore infinite. This contradicts the fact that $X$ is $(\dim Y,1)$-bounded. We conclude that almost all components of $\underline{\Hom}_k(Y,X)$ have dimension $<\dim X$.  
\end{proof}

\begin{remark}  \label{remark:impls}
Let $n\geq 1$ and $m\geq 1$ be positive integers. Let $X$ be a projective scheme over $k$. We have shown the following statements (see also \cite[Chapter~12]{JBook}).
\begin{itemize}
\item If $X$ is algebraically hyperbolic over $k$, then  $X$ is  bounded over $k$.
\item The scheme $X$ is bounded over $k$ if and only if  $X$ is $1$-bounded over $k$.
\item The scheme $X$ is $(n,m)$-bounded over $k$ if and only if $X$ is $(1,1)$-bounded over $k$. 
\item If $X$ is $1$-bounded over $k$, then $X$ is $(n,m)$-bounded over $k$.
\item If $X$ is $(n,m)$-bounded over $k$, then $X$ is groupless over $k$.
\end{itemize}
In a diagram, taking into account Demailly's theorem and Brody's Lemma when $k=\CC$, our results can be summarized as follows (with $n\geq 1$ and $m\geq 1$ below):
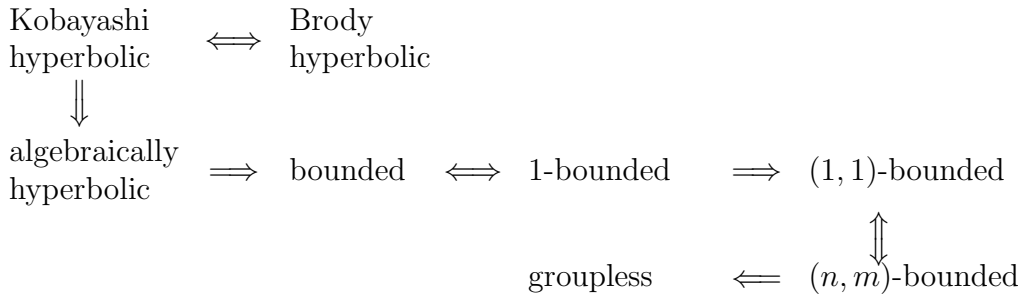
\begin{center}
\begin{tabular}{m{2.25cm}cm{1.7cm}cm{2.35cm}cm{3.0cm}c}
Kobayashi hyperbolic &$\Longleftrightarrow$ &Brody \mbox{hyperbolic} &\\[0.2cm]
\qquad \rotatebox[origin=c]{270}{$\Longrightarrow$}\\[0.2cm]
algebraically hyperbolic &$\Longrightarrow$ &bounded &$\Longleftrightarrow$ &1-bounded &$\Longrightarrow$ & $(1,1)$-bounded   &\\[0.2cm]
&&&&&&\qquad \rotatebox[origin=c]{270}{$\Longleftrightarrow$}\\[-0.1cm]
&&    &   &   groupless & $\Longleftarrow$ &  $(n,m)$-bounded
\end{tabular}
\end{center}\mbox{}\\[0.0cm]
 The conjectures of Demailly, Green-Griffiths, and Lang predict that all of these implications are actually equivalences.
\end{remark}

    \section{Proofs of main results} \label{S10}
    In this section we prove the results on algebraic hyperbolicity and $1$-bounded varieties.
    \subsection{Algebraically hyperbolic varieties}
We  prove all the results on algebraic hyperbolicity stated in Section \ref{section:main_results}. The proofs are     applications and combinations of all our results.
  
  \begin{proof}[Proof of Theorem \ref{thm:geometricity}]
  This is Theorem \ref{thm1}. 
  \end{proof}

 \begin{proof}[Proof of Theorem \ref{thm:aut_end}]
  Since algebraically hyperbolic projective varieties are bounded (Theorem \ref{alg-hyp-bounded}), there are only finitely many dominant rational maps   $Y\dashrightarrow X$ by Proposition \ref{dom-fin-norm}. The rest of the theorem follows from Corollary \ref{Aut-fin-bound}.  
  \end{proof}
  
     \begin{proof}[Proof of Theorem \ref{thm:pointed_hom}]
  Since algebraically hyperbolic projective varieties are bounded (Theorem \ref{alg-hyp-bounded}), the theorem follows from Corollary \ref{Hom-subsch-finite}.   
  \end{proof}
  
   \begin{proof}[Proof of Theorem \ref{thm:projectivity_of_hom}]
  Let $X$ be a projective algebraically hyperbolic scheme over $k$ and let $Y$ be a projective normal scheme over $k$. Since $X$ is bounded (Theorem \ref{alg-hyp-bounded}), the scheme $\underline{\Hom}_k(Y,X)$ is  an algebraically hyperbolic projective scheme over $k$ with $\dim \underline{\Hom}_k(Y,X) \leq \dim X$ (Corollary \ref{cor:projectivity_of_Homs}). 
  
    To see that $\dim \underline{\Hom}_k^{nc}(Y,X) < \dim X$, let $Z\subset \underline{\Hom}_k^{nc}(Y,X)$ be a reduced irreducible component with $\dim Z = \dim X$. For all $y$ in $Y(k)$, consider the evaluation map $\mathrm{eval}_y\colon Z\to X$, and note that it is finite (as shown in the proof of Corollary \ref{cor:projectivity_of_Homs}). Since $\dim Z = \dim X$, for all $y$ in $Y(k)$, the finite morphism $\mathrm{eval}_y$ is surjective. Thus, as $\mathrm{Sur}_k(Z,X)$ is finite (Theorem \ref{thm:aut_end}), there exist an integer $n\geq 1$ and  points $y_1,\ldots, y_n\in Y(k)$ such that, for all $y$ in $Y(k)$, we have that $\mathrm{eval}_y\in \{ \mathrm{eval}_{y_1},\ldots,\mathrm{eval}_{y_n}\}$.  In other words, every morphism $f\colon Y\to X$ in $Z$ takes on only finitely many values (namely $f(y_1),\ldots, f(y_n)$). In particular, since $Z$ is irreducible, we conclude that every $f$ in $Z$ takes on precisely one value, i.e., $f$ is constant. This contradicts the fact that $Z\subset \underline{\Hom}^{nc}(Y,X)$.   
  \end{proof}

  \subsection{Bounded varieties}
  We prove all the results on $1$-bounded varieties stated in Section \ref{section:further_results}.
  
  \begin{proof}[Proof of Theorem \ref{thm:geometricity2}]
  This is Theorem \ref{thm:1b_is_geom}.
  \end{proof}
  
    \begin{proof}[Proof of Theorem \ref{thm:pointed_hom2}]
   Since $1$-bounded projective    projective varieties are bounded (Theorem \ref{thm:1_implies_n}), the theorem follows from Corollary \ref{Hom-subsch-finite}. 
  \end{proof}
  
    \begin{proof}[Proof of Theorem  \ref{thm:projectivity_of_hom2}] (We follow very closely the proof  of Theorem \ref{thm:projectivity_of_hom}.)
  Let $X$ be a projective $1$-bounded scheme over $k$ and let $Y$ be a projective normal scheme over $k$. Since $X$ is bounded (Theorem \ref{thm:1_implies_n}), the scheme $\underline{\Hom}_k(Y,X)$ is  a bounded projective scheme over $k$ with $\dim \underline{\Hom}_k(Y,X) \leq \dim X$ (Corollary \ref{cor:projectivity_of_Homs}). 
  
    To see that $\dim \underline{\Hom}_k^{nc}(Y,X) < \dim X$, let $Z\subset \underline{\Hom}_k^{nc}(Y,X)$ be a reduced irreducible component with $\dim Z = \dim X$. We now redo the argument in the proof of Theorem \ref{thm:projectivity_of_hom2}. For all $y$ in $Y(k)$, consider the evaluation map $\mathrm{eval}_y\colon Z\to X$, and note that it is finite (as shown in the proof of Corollary \ref{cor:projectivity_of_Homs}). Since $\dim Z = \dim X$, for all $y$ in $Y(k)$, the finite morphism $\mathrm{eval}_y$ is surjective. Thus, as $\mathrm{Sur}_k(Z,X)$ is finite (Theorem \ref{thm:pointed_hom2}), there exist an integer $n\geq 1$ and  points $y_1,\ldots, y_n\in Y(k)$ such that, for all $y$ in $Y(k)$, we have that $\mathrm{eval}_y\in \{ \mathrm{eval}_{y_1},\ldots,\mathrm{eval}_{y_n}\}$.  In other words, every morphism $f\colon Y\to X$ in $Z$ takes on only finitely many values (namely $f(y_1),\ldots, f(y_n)$). In particular, since $Z$ is irreducible, we conclude that every $f$ in $Z$ takes on precisely one value, i.e., $f$ is constant. This contradicts the fact that $Z\subset \underline{\Hom}^{nc}(Y,X)$.  
  \end{proof}
 
\begin{proof}[Proof of Theorem \ref{thm:boundedness_to_uniform}]
 Let $g\geq 2$ be an integer and let $\mathcal{U}\to \mathcal{M}$ be the universal smooth proper curve of genus $g$ over the moduli stack  $\mathcal{M}:=\mathcal{M}_{g,k}$ of smooth proper genus $g$ curves over $k$. Let $\overline{\mathcal{U}}\to \overline{\mathcal{M}}$ be the universal stable curve over the moduli stack $\overline{\mathcal{M}}$ of stable curves of genus $g$.  Let $Z$ be a normal projective bounded scheme over $k$ and let $Z\to \overline{\mathcal{M}}$ be a flat   surjective morphism. Let $Y:= \overline{\mathcal{U}} \times_{\overline{M}} Z$ and consider the morphism $Y\to Z$. 
 
 Now, since $X$ is $1$-bounded over $k$, it follows from Theorem \ref{thm:1_implies_n}   that $X$ is bounded over $k$. Therefore, the scheme $\underline{\Hom}_k(Y, X\times Z)$ is of finite type over $k$. In particular, the morphism $\underline{\Hom}_{Z}(Y,X\times Z)\to Z$ is of finite type. By descent, the morphism $\underline{\Hom}_{\overline{\mathcal{M}}}(\overline{\mathcal{U}},X\times \overline{\mathcal{M}})\to \overline{\mathcal{M}}$ is of finite type. By base-change, the morphism $$\underline{\Hom}_{\mathcal{M}}(\mathcal{U}, X\times \mathcal{M}) \to \mathcal{M}$$ is of finite type. As  the latter morphism  is of finite type (for every $g\geq 2$) we see that, for every ample line bundle $\mathcal{L}$ on $X$ and every integer $g\geq 2$, there is an integer $\alpha(X, \mathcal{L}, g)$ such that, for every smooth projective connected curve $C$ of genus $g$ over $k$ and every morphism $f:C\to X$, the inequality 
 \[
 \deg_C f^\ast \mathcal{L} \leq \alpha(X,\mathcal{L},g)
 \] holds. This implies the desired statement and concludes the proof.
\end{proof}

 Note that Theorem \ref{thm:boundedness_to_uniform} says that a projective variety over $k$ is $1$-bounded over $k$ if and only if it is \emph{weakly bounded} in the sense of Kov\'acs-Lieblich \cite{KovacsLieblichErratum}.

\section{Conjectures related to Demailly's  and Green--Griffiths--Lang's conjecture} \label{section:conjectures}

The following conjecture is a consequence of Demailly's conjecture (Conjecture \ref{conj:demailly}), and thus also Green--Griffiths--Lang's conjecture \cite{Lang}.   The conjecture says that  the total space of a family of projective algebraically hyperbolic varieties over a projective algebraically hyperbolic base variety is algebraically hyperbolic.
 
\begin{conjecture}[Fibration property]\label{conj:families} Let $k$ be  an algebraically closed field of characteristic zero. 
Let $f:X\to Y$ be a   morphism of projective varieties  over $k$. If $Y$ is algebraically hyperbolic over $k$, and,  for all $y$ in $Y(k)$, the projective scheme $X_y$ is algebraically hyperbolic over $k$, then  $X$ is algebraically hyperbolic over $k$. 
\end{conjecture}

The analogue of this conjecture for projective families of Kobayashi hyperbolic   varieties is known and follows from \cite[Corollary~3.11.2]{Kobayashi}. We now explain how Conjecture \ref{conj:families} follows from Demailly's conjecture (Conjecture \ref{conj:demailly}).

\begin{remark}[Demailly's conjecture implies Conjecture \ref{conj:families}]
To see that Conjecture \ref{conj:families} is a consequence of Demailly's conjecture (Conjecture \ref{conj:demailly}),  we may and do assume that $k\subset \CC$. Then,  with the notation as in Conjecture \ref{conj:families}, by the geometricity of algebraic hyperbolicity (Theorem \ref{thm:geometricity}), the fibers of the morphism $X_\CC\to Y_\CC$ are algebraically hyperbolic over $\CC$ and $Y_\CC$ is algebraically hyperbolic over $\CC$. By Demailly's conjecture, for every $t$ in $Y(\CC)$, the fiber $X_{y}$ is Kobayashi hyperbolic (as a complex analytic space) and the projective variety $Y_\CC$ is Kobayashi hyperbolic. Therefore, $X_\CC$ is  Kobayashi hyperbolic \cite[Corollary~3.11.2]{Kobayashi}.  Since  Kobayashi hyperbolic projective varieties over $\CC$ are algebraically hyperbolic (Theorem \ref{thm:dem1}), this shows that $X_\CC$ is algebraically hyperbolic over $\CC$. We conclude that $X$ is algebraically hyperbolic over $k$.  
\end{remark}

The following conjecture relates all notions of ``boundedness'' introduced in this paper (see Section \ref{section:boundedness}). 
\begin{conjecture}\label{conj:boundedness}
Let $k$ be an algebraically closed field of characteristic zero and let $X$ be a projective variety over $k$. Then the following are equivalent.
\begin{enumerate}
\item The projective variety $X$ is algebraically hyperbolic over $k$.
\item The projective variety $X$ is bounded over $k$.
\item For all $n\geq 1$, the projective variety $X$ is $(n,1)$-bounded.
\end{enumerate}
\end{conjecture}
 
 Note that $(1)\implies (2)$ is Theorem \ref{alg-hyp-bounded} and that $(2)\implies (3)$ is Remark \ref{bounded-impications}. Other relations between the three notions in Conjecture \ref{conj:boundedness} are summarized in Remark \ref{remark:impls}. The implication $(3) \implies (2)$ is currently not known and neither is the implication $(2)\implies (1)$. To show that  $(3)\implies (2)$, it suffices to show that, if $X$ is $(n,1)$-bounded for all $n\geq 1$, then $X$ is $1$-bounded.

 We conclude by noting that  the implication $(3)\implies (1)$ in Conjecture \ref{conj:boundedness} is a consequence of Green--Griffiths--Lang's conjecture in \cite{Lang}. Indeed, $(1,m)$-bounded varieties are groupless by Proposition \ref{prop:nm_bounded_is_groupless}, and it follows from Green--Griffiths--Lang's conjectures that projective groupless varieties are algebraically hyperbolic.
 
\bibliography{refs}{}
\bibliographystyle{plain}

\end{document}